\documentclass[11pt, twoside]{amsart}
\usepackage{amssymb,amsthm,amsmath}\usepackage{txfonts} 
\usepackage[english]{babel}
\theoremstyle{plain} \hoffset -2 cm \voffset -.8 cm \textwidth 16
cm \textheight 22 cm 

\newtheorem{teo}{Theorem}
\newtheorem{oss}[teo]{Remark}
\newtheorem{Prop}[teo]{Proposition}
\newtheorem{war}[teo]{Warning}
\newtheorem{lemma}[teo]{Lemma}

\newtheorem{Defi}[teo]{Definition}
\newtheorem{es}[teo]{Example}
\newtheorem{corollario}[teo]{Corollary}
\newtheorem{no}[teo]{Notation}

\newcommand{\cc}{_{^{_\HH}}}
\newcommand{\vv}{_{^{_\VV}}}

\newcommand{\res}{\mathop{\hbox{\vrule height 7pt width .5pt depth 0pt
\vrule height .5pt width 6pt depth 0pt\,}}\nolimits}

\def \ope{^\bot}

\def \ot{^\top}

\def \op{^\perp}

\def \DH{h}

\newcommand{\LL}{\mathop{\hbox{\vrule height .5pt width 6pt depth
0pt \vrule height 7pt width .5pt depth 0pt\,}}\nolimits}
\newcommand{\rr}{_{^{_\mathit{R}}}}

\newcommand{\ngr}{_{^{_{\mathrm Gr}}}}

\newcommand{\ci}{_{^{_{\HH_i}}}}

\newcommand{\cd}{_{^{_{\HH_2}}}}
\newcommand{\cu}{_{^{_{\HH_1}}}}

\def \cont{{\mathbf{C}}}
\def \cin{{\mathbf{C}^{\infty}}}

\def\dim {\mathrm{dim}}
\def\dc {d_{CC}}

\def\ss{_{^{_{\HS}}}}
\def\ts{_{^{_{\TS}}}}

\def\x{x}
\def\g{\mathit{g}\cc}

\def\dg{\mathit{grad}\cc}
\def\qq{\mathit{grad}\ss}

\def \nis {\sigma^{n-2}\cc}

\def \per {\sigma^{n-1}\cc}

\def\SC{{C}}
\def\nm{\nu}

\def\UU{\mathcal{U}}

\def \nn{\nu\cc}

\def \XH{\mathfrak{X}(\HH)}
\def \XX{\mathfrak{X}}

\def \MS{\mathcal{H}\cc}

\def \P{{\mathcal{P}}}

\def \PH{\P\cc}

\def \Om{\Omega}

\def \ee{\mathrm{e}}

\def \Omn{\sigma^n\rr}
\def \R{\mathbb{R}}
\def \Rn{\mathbb{R}^{\DN}}

\def \div{\mathit{div}}

\def \GG{\mathbb{G}}
\def \gg{\mathfrak{g}}

\def\divh{\div\cc}

\def\lh{\mathcal{D}\ss}

\def\lll{\mathcal{L}\ss}

\def\tsc{\nabla^{^{_{\TT\!{S}}}}}
\def\gs{\nabla^{_{\HS}}}
\def\gc{\nabla^{_{\HH}}}

\def\UU{\mathcal{U}}

\def\UU{\mathcal{U}}

\def \nn{\nu_{_{\!\HH}}}
\def \XG{\mathfrak{X}(\GG)}

\def \ee{\mathrm{e}}

\def \Om{\Omega}
\def \Rn{\mathbb{R}^{n}}
\def \R{\mathbb{R}}

\def \cji {c_{j\,i}(x)}
\def \C { C(x):=[\cji]_{j,i},\,\, {j=1,\ldots,m \,,\, i=1,\ldots,n}}

\def \Qdim {Q:=\sum_{i=1}^{k}i\,\DH_i}

\def \X {X=(X_{1}, \ldots, X_{m_1})}

\def \X0 {X_{1}(0)\!=\!\partial_{x_{1}}, \ldots, X_{m_1}(0)\!=\!\partial_{x_{m_1}}}

\def \HG {\HH\GG}
\def \HS {\HH S}

\def \TG {\mathit{T}\GG}
\def \HH {\mathit{H}}

\def \VV {\mathit{V}}
\def \TT {\mathit{T}}
\def \TS {\mathit{T} S}

\def \grad{\mathit{grad}}
\def \C0H{\mathbf{C}_{0}^{\infty}(U,\HG)}
\def \C00{\mathbf{C}_{0}^{\infty}(U)}
\def \C01{\mathbf{C}_{0}^{1}(U)}

\def \L1{d\,\mathcal{L}^n}

\def \H1{\mathcal{H}_{{\bf cc}}^{1}}

\def \Ar{\mathcal{H}^{n-1}_{Eu}}

\def \Vol{{\mathcal Vol}^{n}}

\def \exp{\textsl{exp\,}}

\def \Om{\Omega}
\def \Rn{\mathbb{R}^{n}}
\def \R{\mathbb{R}}

\def \cji {c_{j\,i}(x)}
\def \C { C(x):=[\cji]_{j,i},\,\, {j=1,\ldots,m \,,\, i=1,\ldots,n}}

\def \GG{\mathbb{G}}
\def \gg{\mathfrak{g}}

\def \X {X=(X_{1}, \ldots, X_{m_1})}

\def \X0 {X_{1}(0)\!=\!\partial_{x_{1}}, \ldots, X_{m_1}(0)\!=\!\partial_{x_{m_1}}}

\def \HG {\mathit{H}}

\def \C0H{\mathbf{C}_{0}^{\infty}(\Om,\HG)}
\def \C00{\mathbf{C}_{0}^{\infty}(\Om)}
\def \C01{\mathbf{C}_{0}^{1}(\Om)}

\def \exp{\textsl{exp\,}}

\def\GG{\mathbb{G}}

\begin{document}

\vskip 3cm
\begin{center}
{\bf \LARGE{Geometric inequalities in  Carnot groups}}

\vskip 1cm {\large Francescopaolo Montefalcone\footnote{F. M. has been partially supported by the Fondazione CaRiPaRo Project ``Nonlinear Partial Differential Equations: models, analysis, and control-theoretic problems".}\footnote{The authors wish to thank the anonymous referee for helpful comments that improved the paper.} }

\end{center}

\markboth{Francescopaolo Montefalcone}{Geometric inequalities in
Carnot groups}

\section*{Abstract} \small
Let $\GG$  be a sub-Riemannian $k$-step Carnot group of homogeneous dimension
$Q$. In this paper, we shall prove  several geometric inequalities concerning smooth hypersurfaces (i.e. codimension one submanifolds) immersed in $\GG$, endowed with the $\HH$-perimeter measure.
\\{\noindent \scriptsize \sc Key words and phrases:}
{\scriptsize{\textsf {Carnot groups; Sub-Riemannian geometry;
hypersurfaces; geometric inequalities.}}}\\{\scriptsize\sc{\noindent Mathematics Subject
Classification:}}\,{\scriptsize \,\,49Q15, 46E99, 43A80.}

\normalsize

\tableofcontents

\section{Introduction}

During the last years there was an increasing interest in studying
Analysis and Geometric Measure Theory in  metric spaces (see
\cite{A2} \cite{AK1, AK2}, \cite{8}, \cite{DaSe}, \cite{GN},
\cite{33} and bibliographic references therein,  but   this list
is far from being exhaustive). In this regard, important examples
of highly non-Euclidean geometries are represented by the
so-called Carnot-Charath\'{e}odory (or sub-Riemannian) geometries;
see \cite{6}, \cite{18}, \cite{28}, \cite {29, 30, 31},
\cite{Stric}, \cite{Ver}. In this context, Carnot groups  play
the role of modeling the tangent space (in a suitable generalized
sense, which is related with the Gromov-Hausdorff convergence) of
a sub-Riemannian manifold; see \cite{18}, \cite{28}. For this and
many other reasons, Carnot groups are an intriguing field of
research; see \cite{3}, \cite{4}, \cite{5}, \cite{7}, \cite{9},
\cite{12, 13}, \cite{14, 15, 16, 17}, \cite{HP}, \cite{19, 21},
\cite{Montea, 22}, \cite{32}.

A $k$-{\it{step Carnot group}}
$(\GG,\bullet)$ is an $n$-dimensional, connected, simply
connected, nilpotent, stratified Lie group (with respect to the
group multiplication $\bullet$) whose Lie algebra $\gg\cong\Rn$
satisfies:\[ {\mathfrak{g}}={\HH}_1\oplus...\oplus {\HH}_k,\quad
 [{\HH}_1,{\HH}_{i-1}]={\HH}_{i}\quad(i=2,...,k),\,\,\,
 {\HH}_{k+1}=\{0\}.\]We assume that $\DH_i=\dim\HH_i \,(i=1,...,k)$ so that $n=\sum_{i=1}^{k}\DH_i$. Any Carnot group $\GG$ has a 1-parameter
family of dilations, adapted to the stratification, that makes it
a {\it homogeneous group}, in the sense of Stein's definition; see
\cite{Stein}. We refer the reader to Section \ref{cg} for a more detailed introduction to Carnot groups.

In this paper, we shall prove  some geometric inequalities concerning smooth hypersurfaces  immersed in a sub-Riemannian $k$-step Carnot group $\GG$ of homogeneous dimension $\Qdim$. We have to stress that hypersurfaces will be  endowed  with the so-called $\HH$-perimeter measure $\per$, which is a natural substitute of the intrinsic $(Q-1)$-dimensional CC Hausdorff measure.
In Section \ref{hgxzc}, we will discuss some preliminaries notions concerning  homogeneous measures and the horizontal geometry of hypersurfaces. Then we will recall some tools which will be important in the sequel, such as a Coarea-type formula and the horizontal integration by parts theory; see Section \ref{othy}.

In Section \ref{poincinsect} we will extend to this setting some Isoperimetric-type Constants, introduced by Cheeger in the Seventies for compact Riemannian manifolds in \cite{Cheeger} and later studied by Yau in \cite{Yau}.

In particular, we shall prove the validity of some
global inequalities for smooth compact
hypersurfaces with (or without) boundary, immersed into $\GG$.
Here, we would like to remark that there is a strong relationship among these inequalities and
some eigenvalue  problems related to the 2nd order differential
operator $\lll$ (which is nothing but a horizontal version of the Laplace-Beltrami operator); see, more precisely, Definition \ref{epr} in Section \ref{hgxzc}.

Roughly speaking, after defining the
isoperimetric constants (in purely geometric terms), we will show
that they are equal to the infimum of some  Rayleigh's quotients. More precisely, let $S\subset\GG$ be a smooth hypersurface
and assume  $\partial S\neq
\emptyset$. Furthermore, set
$${\rm Isop}(S):=\inf\frac{\nis(N)}{\per(S_1)},$$where
$N\subset S$ is a smooth hypersurface of $S$ such that 
$N\cap \partial S=\emptyset$ and $S_1$ is the unique
$(n-2)$-dimensional submanifold of $S$ such that  $N=\partial
S_1$. We have to stress that $\per$ and $\nis$ denote homogeneous measures on $S_1$ and $N$, respectively. These measures can be thought of, respectively,  as the $(Q-1)$-dimensional and the $(Q-2)$-dimensional CC Hausdorff measures on $S_1$ and $N$; see Section \ref{hgxzc}. Then, it will be shown that $${\rm
Isop}(S)=\inf\frac{\int_S|\qq \psi|\per}{\int_S|\psi|\per},$$
where the infimum is taken over suitably smooth functions on
$S$ such that $\psi|_{\partial S}=0$. As mentioned, this constant is related to
the first non-zero eigenvalue $\lambda_1$ of the following
Dirichlet-type problem:
\begin{displaymath}
\left\{%
\begin{array}{ll}
   -\lll\psi=\lambda\, \psi,\\\quad\, \psi|_{\partial S} =0; \\
\end{array}%
\right.\end{displaymath}see Definition \ref{epr}. Indeed, we shall
see that \[\lambda_1\geq \frac{\left({\rm Isop}(S)\right)^2}{4};\]
see Corollary \ref{555}. Some similar results concerning another
isoperimetric constant will be proved; see Theorem \ref{zdaz} and
Corollary \ref{1zdaz}. The proofs of these results follow the
scheme of the Riemannian case, for which we refer the reader to
Yau, \cite{Yau}; see also \cite{Cheeger} and \cite{Ch1, Ch3}. We
also remark that the main technical tool in the original proofs is
the Coarea formula.

In Section \ref{ChRe} we shall prove two geometric inequalities involving volume, $\HH$-perimeter and  the 1st eigenvalue of the operator $\lll$ on $S$. These results generalize an inequality of Chavel (see \cite{88}) and an inequality of Reilly (see \cite{222}), respectively.

In Section \ref{wlinisoineq} we will extend to the
Carnot setting some classical differential-geometric results (such as linear isoperimetric inequalities); see,
for instance,  \cite{b}
 and references therein.
 The starting point is an integral formula very similar to the Euclidean  Minkowsky
 Formula; see  Corollary \ref{Mink} for a precise statement.  In particular, we will
 show that\[(\DH-1)\,\per(S)\leq
R \left( \int_S \left(|\MS|+ |C\cc\nn|\right)\,\per+ \nis(\partial
S) \right) \]where $S\subset\GG$ is a compact hypersurface
with boundary and $R$ denotes the radius of a homogeneous $\varrho$-ball
circumscribed about $S$. From this linear (isoperimetric) inequality, it is possible to infer some geometric
 consequences and, among them, we  prove a weak monotonicity
 inequality for the $\HH$-perimeter; see Section \ref{wmf}, Proposition \ref{in}.

Section \ref{heinz} contains a theorem about non-horizontal graphs in 2-step Carnot groups.
This generalizes a classical result of  Heinz \cite{112}; see also Chern, \cite{933}.

Let us describe this result in the simpler case of the Heisenberg group $\mathbb H^1$. So let
  $S\subset\mathbb H^1$ be a $T$-graph  associated with a function  $t=f(x, y)$ of class
$\mathbf{C}^2$ over the $xy$-plane. If the horizontal mean
curvature $\MS$ of $S$  satisfies a bound $|\MS|\geq C>0$, then
$$C
\mathcal{H}_{Eu}^2(\P_{xy}(\UU))\leq\mathcal{H}_{Eu}^1(\P_{xy}(\partial\UU))$$for
every $\cont^1$-smooth relatively compact open set $\UU\subset S$,
where $\mathcal{H}_{Eu}^i\,(i=1,2)$ is the usual $i$-dimensional
Euclidean Hausdorff measure and $\P_{xy}$ is the orthogonal
projection onto the $xy$-plane. Hence, taking $\UU:=S\cap
C_r({\mathcal T})$, where $C_r({\mathcal T})$ denotes a  vertical
cylinder of radius $r$ around the $T$-axis of ${\mathbb H}^1$,
yields
$$r\leq \frac{2}{C}$$ for every $r>0$.
 It follows that any entire $xy$-graph of class $\cont^2$, having either constant or only bounded  horizontal mean curvature $\MS$, must be necessarily  a $\HH$-minimal surface. An analogous result holds true in the framework of step 2 Carnot groups; see Theorem \ref{pdt5}.

 In Section
 \ref{wme} we shall study some (local) Poincar\'e-type
 inequalities,
 depending on the local geometry of the hypersurface $S$ and, more precisely, on its characteristic set $C_S$; see Theorem \ref{0celafo}, Theorem \ref{1celafo}.

For instance, let $S\subset\GG$ be a  $\cont^2$-smooth
hypersurface with bounded horizontal mean curvature $\MS$. Then,
we shall prove that for every $x\in S$ there exists $R_0\leq {\rm
dist}_\varrho(x,\partial S)$ (which explicitly depends on $C_S$) such that:
\[
\left(\int_{S_R}|\psi|^p\per\right)^{\frac{1}{p}}\leq C_p\,
R\left(\int_{S_R}|\qq\psi|^p\per\right)^{\frac{1}{p}} \qquad
p\in[1,+\infty[
\]for all
$\psi\in\cont_0^1(S_R)$ and all $R\leq R_0$, where $S_R:=S\cap
B_\varrho(x,R)$.

 These
 results are obtained by means of elementary ``linear'' estimates starting from the
  horizontal integration by parts formula, together with a
 simple analysis of the role played by the characteristic set.
Finally, in Section \ref{Cacciopp} we will prove the validity of a Caccioppoli-type inequality for  weak solutions of the operator $\lll$.

\subsection{Carnot groups}\label{cg}

A $k$-{\it{step Carnot group}}
$(\GG,\bullet)$ is a finite-dimensional connected, simply connected, nilpotent and
stratified Lie group with respect to a polynomial group law $\bullet$.
The Lie algebra $\gg \cong\Rn$
fulfils the conditions: ${\mathfrak{g}}={\HH}_1\oplus...\oplus
{\HH}_k$,  $[{\HH}_1,{\HH}_{i-1}]={\HH}_{i}\quad\forall\,\,i=2,...,k+1$, ${\HH}_{k+1}=\{0\},$ where  $[\cdot,\cdot]$ denotes the Lie brackets and each $\HH_i$ is a vector subspace of $\gg$.
In particular, we denote by $0$ the identity of $\GG$ and assume that  $\gg\cong\TT_0\GG$. 
We also use the notation
 $\HH:=\HH_1$ and
${\VV}:={\HH}_2\oplus...\oplus {\HH}_k.$
The subspaces $\HH$ and $\VV$ are smooth subbundles
 of $\TG$ called
{\it horizontal}  and {\it vertical}, respectively. 
\begin{no}\label{1notlne0}Throughout this paper, we  denote by $\P\ci:\TG\longrightarrow\HH_i$ the orthogonal projection map from $\TG$ onto $\HH_i$ for any $i=1,...,k$. In particular, we set $\P\cc:=\P\ci$. Analogously, we  set $\P\vv:\TG\longrightarrow\VV$ to denote the orthogonal projection map from $\TG$ onto $\VV$.
\end{no}

Let
$\DH_i:=\dim{{\HH}_i}$ for any $i=1,...,k$. Set $n_0:=0$ and $n_i:=\sum_{j=1}^{i}\DH_j$ for any $i=1,...,k$. Note that $n_1=\DH_1$, $n_2=\DH_1+\DH_2$,..., and $n_k=n$. 

\begin{no}\label{1notlne1}Throughout this paper, we  set $I\ci:=\{n_{i-1}+1,...,n_i\}$ for any $i=1,...,k$. We also set $I\vv:=\{h_1+1,...,n\}$ 
and use Greek letters
 $\alpha, \beta, \gamma,...,$ for indices in $I\vv$. For the sake of simplicity, we  set $\DH:=\DH_1$ and $I\cc:=I\cu$.
\end{no}

The horizontal
bundle $\HH$ is generated by a frame
$\underline{X\cc}:=\{X_1,...,X_{\DH}\}$ of left-invariant vector
fields. This frame can be completed to a global  graded,
left-invariant frame $\underline{X}:=\{X_1,...,X_n\}$ for $\TG$.
Note that the standard basis $\{\ee_i:i=1,...,n\}$ of $\Rn$
can be relabeled to be {\it graded} or {\it adapted to the
stratification}. Any left-invariant vector field of the frame
$\underline{X}$ is given by
${X_i}(x)={L_x}_\ast\ee_i\,(i=1,...,n)$, where ${L_x}_\ast$
denotes the differential of the left-translation $L_x$, defined by ${L_x}y:=x\bullet y\,\,\,\forall\,\,y\in\GG$. We also
fix a Euclidean metric on $\gg=\TT_0\GG$ such that $\{\ee_i:
i=1,...,n\}$ is an orthonormal  basis. This metric $g=\langle\cdot,\cdot\rangle$ extends to the
whole tangent bundle  by left-translations and makes
$\underline{X}$ an
  orthonormal  left-invariant frame. Therefore $(\GG, g)$
 is a Riemannian manifold.

 Let $\exp:\gg\longrightarrow\GG$ be the exponential map. Hereafter, we will use  {\it exponential coordinates of the 1st
kind}; see \cite{Varad}, Ch. 2, p. 88.

As for the case of nilpotent Lie groups, the 
multiplication $\bullet$ of $\GG$ is uniquely determined by the ``structure'' of the
 Lie algebra $\gg$. This is the content of the {\it Baker-Campbell-Hausdorff
formula}; see \cite{Corvin}. More precisely, one has
$$\exp(X)\bullet\exp(Y)=\exp(X\star Y)\qquad\forall\,\,X,\,Y \in\gg,$$ where
${\star}:\gg \times \gg\longrightarrow \gg$ denotes the so-called
 {\it Baker-Campbell-Hausdorff product} given by \begin{eqnarray}\label{CBHformula}X\star Y= X +
Y+ \frac{1}{2}[X,Y] + \frac{1}{12} [X,[X,Y]] -
 \frac{1}{12} [Y,[X,Y]] + \mbox{ brackets of length} \geq 3.\end{eqnarray}

Using exponential coordinates, \eqref{CBHformula}  the
group multiplication $\bullet$ turns out to be polynomial and
explicitly computable; see \cite{Corvin}. Moreover,
$0=\exp(0,...,0)$ and the inverse of $x\in\GG$
 $(x=\exp(x_1,...,x_{n}))$ is ${x}^{-1}=\exp(-{x}_1,...,-{x}_{n})$.

A {\it sub-Riemannian metric} $g\cc$ is a symmetric positive
bilinear form on the horizontal bundle $\HH$. The {\it
{CC}-distance} $\dc(x,y)$ between $x, y\in \GG$ is given by
$$\dc(x,y):=\inf \int\sqrt{g\cc(\dot{\gamma},\dot{\gamma})}\,dt,$$
where the infimum is taken over all piecewise-smooth horizontal
paths $\gamma$ joining $x$ to $y$. Later, we shall choose
$\g:=g_{|\HH}$.

Carnot groups are {\it homogeneous groups}, that is, they admit a
1-parameter group of
automorphisms
$\delta_t:\GG\longrightarrow\GG$ $(t\geq 0)$ defined by
$\delta_t x
:=\exp\left(\sum_{j,i_j}t^j\,x_{i_j}\ee_{i_j}\right)$,
where $x=\exp\left(\sum_{j,i_j}x_{i_j}\ee_{i_j}\right)\in\GG.$ As already said, the
{\it homogeneous dimension} of $\GG$ is the integer
 $\Qdim$
coinciding with the {\it Hausdorff dimension} of $(\GG,\dc)$ as a
metric space; see \cite{28}.

We recall that a continuous distance
$\varrho:\GG\times\GG\longrightarrow\R_+\cup\{0\}$  is a \it homogeneous distance \rm if, and only if,
\begin{eqnarray*}\varrho(x,y)=\varrho(z\bullet x,z\bullet
y)\quad\forall\,x,\,y,\,z\in\GG;\qquad
\varrho(\delta_tx,\delta_ty)=t\varrho(x,y)\quad\forall t\geq
0.\end{eqnarray*}

The {\it structural constants} of $\gg$ (see \cite{Ch1}) associated with the frame
$\underline{X}$ are defined by
 $\SC^r_{ij}:=\langle [X_i,X_j],
 X_r\rangle$ $\forall\, i,j, r=1,...,n$.
They are skew-symmetric and satisfy Jacobi's identity. The stratification of
the Lie algebra $\gg$  implies a fundamental ``structural'' property of
Carnot groups, i.e. if $ X_i\in {\HH}_{l},\, X_j \in {\HH}_{m}$,
then $[X_i,X_j]\in {\HH}_{l+m}$. It is worth remarking that, if $i\in
I{\!_{^{_{{\HH}_s}}}}$ and $j\in I{\!_{^{_{{\HH}_r}}}}also  $, then
\begin{equation}\label{notabilmente}\SC^m_{ij}\neq 0  \Longrightarrow m \in I{\!_{^{_{{\HH}_{s+r}}}}}.\end{equation}
 Equivalently, if $\SC^r_{ij}\neq 0$, then ${\rm ord}(i)+{\rm ord}(j)={\rm ord}(r)$, where  ${\rm ord}:\{1,...,n\}\rightarrow\{1,...,k\}$ is the function defined as ${\rm ord}(l)=i\Longleftrightarrow l\in I\ci$. 
\begin{no}\label{noconstmatr}Henceforth, we shall set
\begin{itemize}\item$C^\alpha\cc:=[\SC^\alpha_{ij}]_{i,j=1,...,\DH}\in\mathcal{M}_{\DH\times
\DH}(\R) \qquad \forall\,\,\alpha\in I\cd=\{\DH+1, ..., \DH+\DH_{2}\}$;\item
$C^\alpha:=[\SC^\alpha_{ij}]_{i, j=1,...,n}\in
\mathcal{M}_{n\times n}(\R) \qquad \forall\,\,\alpha\in I\vv=\{\DH+1, ...,
n\}$.\end{itemize}\end{no}
\begin{oss}\label{ossconst}It is important to observe that \eqref{notabilmente} immediately implies that the matrices just defined are the only ones which can be non zero.\end{oss}

Let us define the left-invariant co-frame $\underline{\omega}:=\{\omega_i:i=1,...,n\}$ dual to
$\underline{X}$, i.e. $\omega_i=X_i^\ast$ for every $i=1,...,n$.
The {\it left-invariant 1-forms} $\omega_i$ for $i=1,...,n$ are
uniquely determined by the condition
$\omega_i(X_j)=\langle X_i,X_j\rangle=\delta_i^j\,\,\forall\,\,i,
j=1,...,n$,  where $\delta_i^j$ denotes  Kronecker delta.

\begin{Defi}\label{parzconn}
We shall denote by $\nabla$ the  unique  left-invariant
Levi-Civita connection on $\GG$ associated  with the left-invariant metric
$g=\langle\cdot, \cdot\rangle$. Moreover, if
$X, Y\in\XH:=\cin(\GG,\HH)$, we shall set \[\gc_X Y:=\PH(\nabla_X
Y).\]
\end{Defi}

Let  
$\underline{X}=\{X_1,...,X_{n}\}$ be the global left-invariant frame on $\TG$. Then, it turns out that
\begin{equation}\label{milkoszul}\nabla_{X_i} X_j =
\frac{1}{2}\sum_{r=1}^n\left( \SC_{ij}^r  - \SC_{jr}^i +
\SC_{ri}^j\right) X_r\qquad \forall\,\,i,\,j=1,...,n;\end{equation}see, for instance, Milnor's paper \cite{Milnor}, Section 5, pp. 310-311. Furthermore, we stress that $\gc$ is a partial
connection, called {\it horizontal $\HH$-connection};  
see \cite{GE1} or \cite{Koiller}; see also \cite{22} and references therein. Using Definition
\ref{parzconn} together with \eqref{milkoszul} and \eqref{notabilmente}, it is not difficult to show the following:
\begin{itemize}
\item $\gc$ is  flat, i.e.
$$\gc_{X_i}X_j=0\qquad \forall\,\,i,j\in I\cc;$$
 \item 
 $\gc$ is
compatible with the sub-Riemannian metric $g\cc$, i.e.
$$X\langle Y, Z \rangle=\langle \gc_X Y, Z \rangle
+ \langle  Y, \gc_X Z \rangle\qquad \forall\,\, X, Y, Z\in
\XH$$
\item $\gc$ is  torsion-free, i.e.
$$\gc_X Y - \gc_Y X-\PH[X,Y]=0\qquad \forall\,\, X, Y\in \XH.$$
\end{itemize}

\begin{Defi}\label{graddivjac}
If $\psi\in\cin({\GG})$ we define the  horizontal gradient of
$\psi$  as the  unique horizontal vector field $\dg \psi$ such
that $\langle\dg \psi,X \rangle= d \psi (X) = X
\psi$ for every $X\in \XH$.  The  horizontal
divergence of $X\in\XH$, $\divh X$, is defined, at each point
$x\in \GG$, by
$$\divh X(x):= \mathrm{Trace}\left(Y\longrightarrow \gc_{Y} X
\right)(x)\quad(Y\in \HH_x).$$For any $Y=\sum_{j\in I\cc}y_jX_j\in\XH$, we denote by $\mathcal J\cc Y$  the  horizontal Jacobian  matrix of $Y$, i.e. $$\mathcal J\cc Y:=\left[X_i(y_j)\right]_{j, i\in I\cc}.$$
\end{Defi}

\begin{es}[Heisenberg group $\mathbb H^n\,(n\geq 1)$]\label{hng}The Lie algebra
$\mathfrak{h}_n\cong\R^{2n+1}$ of the $n$-th Heisenberg group  $\mathbb H^n$ can be described by means of a left-invariant frame
$ \underline{Z}:=\{X_1,Y_1,...,X_i,Y_i,...,X_n,Y_n,T\}$, where, at each $p=\exp(x_1,y_1,x_2,y_2,...,x_n,y_n, t)\in\mathbb{H}^n$, we have set:
$X_i(p):=\frac{\partial}{\partial x_i} -
\frac{y_i}{2}\frac{\partial}{\partial t}$, $Y_i(p):=\frac{\partial}{\partial y_i} +
\frac{x_i}{2}\frac{\partial}{\partial t}$  for every $i=1,...,n$;
$T(p):=\frac{\partial}{\partial t}$.
One has $[X_i,Y_i]=T$ {for every} $i=1,...,n$, and all other commutators
vanish, so that $T$ is the {\it center} of
$\mathfrak{h}_n$ and
 $\mathfrak{h}_n$ turns out to be a nilpotent and stratified Lie algebra of
step 2, i.e. $\mathfrak{h}_n=\HH\oplus \HH_2$. The
  structural constants of $\mathfrak{h}_n$ are
described by the skew-symmetric $(2n\times 2n)$-matrix
\begin{center}
 $C\cc^{2n+1}:=\left|
\begin{array}{ccccc}
  0 & 1 &    \cdot &  0 &  0 \\
  -1 & 0 &  \cdot &  0 &  0 \\
  \cdot & \cdot & \cdot & \cdot & \cdot \\
  0 & 0 &   \cdot & 0 & 1 \\
  0 & 0 &  \cdot & -1 & 0
\end{array}%
\right| $.\end{center}

\end{es}

\subsection{Hypersurfaces}\label{hgxzc} The (Riemannian)
left-invariant volume form of any Carnot group $\GG$ is defined as
$\sigma^n\rr:=\bigwedge_{i=1}^n\omega_i\in
\bigwedge^n(\TT^\ast\GG)$. By integration of  the $n$-form
$\sigma^n\rr$, one obtains the Haar measure of $\GG$, which equals
the push-forward of the $n$-dimensional Lebesgue measure
$\mathcal{L}^n$ on $\gg\cong\Rn$. The symbols
$\mathcal{H}_{CC}^{s}$, $\mathcal{H}_{Eu}^{s}$ will denote the
intrinsic CC $s$-dimensional Hausdorff measure and the Euclidean
$s$-dimensional Hausdorff measure, respectively. (Sometimes we
will use the notation $\sigma^n\rr=\Vol$). Let $S\subset\GG$ be a
 hypersurface (i.e. a codimension
1 submanifold of $\GG$) of class $\mathbf{C}^i$ $(i\geq 1)$. Let $\nu$ denote the (Riemannian) unit
normal vector along $S$. Then $x\in S$ is a {\it characteristic
point} if and only if $\dim\,\HH_x = \dim (\HH_x \cap \TT_x S)$.
The {\it characteristic set} of $S$ is given by $ C_S:=\{x\in S :
\dim\,\HH_x = \dim (\HH_x \cap \TT_x S)\}$. In other words, a
point $x\in S$ is non-characteristic (hereafter abbreviated as NC)
if  and only if $\HH$ is transversal to $S$ at $x$. Hence, one has
$C_S:=\{x\in S : |\PH\nu(x)|= 0\}$, where $\PH$ denotes
orthogonal projection onto $\HH$. It is of fundamental importance
that  the $(Q-1)$-dimensional CC Hausdorff measure of the
characteristic set $C_S$  vanishes, i.e.
$\mathcal{H}_{CC}^{Q-1}(C_S)=0$; see, for instance, Theorem 6.6.2 in \cite{19}. 
We also stress that if $S$ is a hypersurface of class $\cont^2$, then  precise estimates of the Riemannian Hausdorff dimension of $C_S$  can be found in \cite{Bal3}; see also \cite{4} for the case of the Heisenberg group $\mathbb H^n\, (n\geq 1)$.

The
$(n-1)$-dimensional Riemannian measure along $S$ is defined by
integration of the $(n-1)$-differential form $\sigma^{n-1}\rr\res
S:=(\nu\LL\sigma^{n}\rr)|_{S},$ where  $\LL$ denotes the
``contraction'' operator on differential
 forms; see \cite{c}. We recall that $\LL:
{\bigwedge}^k(\TT^\ast\GG)\rightarrow{\bigwedge}^{k-1}(\TT^\ast\GG)$
is defined, for $X\in\TG$ and
$\alpha\in{\bigwedge}^k(\TT^\ast\GG)$, by setting
$(X \LL \alpha)
(Y_1,...,Y_{k-1}):=\alpha(X,Y_1,...,Y_{k-1})$.

 At each NC point $x\in S\setminus C_S$ the {\it unit $\HH$-normal}  is defined as $\nn:
=\frac{\PH\nu}{|\PH\nu|}$. Similarly to the Riemannian case, we define an $(n-1)$-differential form
 $\per\in\bigwedge^{n-1}(\TT^\ast S)$
by setting $$\per \res S:=(\nn \LL
\sigma^n\rr)|_S.$$  By integration of $\per\res S$, one gets a left-invariant and $(Q-1)$-homogeneous measure, which is called \it $\HH$-perimeter measure. \rm  This measure can be extended
to the whole of $S$ by setting $\per\res C_{S}= 0$. Note that
$\per \res S = |\PH \nu |\,\sigma^{n-1}\rr\,
\res S$. Furthermore, denoting by $\mathcal{S}_{CC}^{Q-1}$ the $(Q-1)$-dimensional
spherical intrinsic CC Hausdorff measure (i.e. associated with the CC-distance $\dc$), then $$\per(S\cap B)=k(\nn)\,\mathcal{S}_{CC}^{Q-1}\res
 ({S}\cap B)\qquad\forall\,\,B\in \mathcal{B}or(\GG),$$where the density-function $k(\nn)$, called {\it metric factor},
explicitly depends on $\nn$ and $\dc$; see \cite{19}.

At each NC point $x\in S\setminus C_S$, the {\it horizontal tangent bundle}
$\HS:=\HH\cap \TS\subset\TS$ and the {\it horizontal normal bundle} $\nn S\subset \HH$
split the horizontal bundle $\HH$ into an orthogonal direct sum,
i.e. $\HH= \nn \oplus \mathit{H}S$. The stratification of $\gg$ induces a
stratification of
$\TS:=\oplus_{i=1}^k\HH_iS,$ where we have set $\HS:=\HH_1S$; see \cite{18}.
Note that at any characteristic point $x\in C_S$ one has $\HH_x=\HH_xS$, so that $$\dim(\HH_xS)=\left\{\begin{array}{ll}\DH-1\quad\,\mbox{if}\,\,x\in S\setminus C_S\\\DH
\,\qquad\,\,\,\, \mbox{if}\,x\in C_S\end{array}\right..
$$

\begin{no}\label{1notlne22}Throughout this paper, we  denote by $\P\ss:\TS\longrightarrow\HS$ the orthogonal projection map from $\TS$ onto $\HS$.
\end{no}

Now let $S\subset\GG$ be a hypersurface of class $\cont^2$ and let
$\tsc$ denote the induced connection on $S$ from $\nabla$. The
tangential connection $\tsc$ induces a partial connection on $\HS$
defined by
$$\gs_XY:=\P\ss\left(\tsc_XY\right) \qquad\,\forall\,\,X,Y\in\XX^1(\HS):=\cont^1(S,\HS).$$
 It turns out that$$\gs_XY=\gc_X Y-\langle\gc_X
Y,\nn\rangle\,\nn\qquad\mbox{for every}\,\,\,X,Y\in\XX^{1}(\HS);$$see \cite{22}.

\begin{Defi}[see \cite{22}]\label{defiophs}
We call $\HS$-{ gradient} of $\psi\in \cont^1({S})$ the unique
horizontal tangent vector field $\qq\psi$ such that
$$\langle\qq\psi,X \rangle= d \psi (X) = X
\psi\qquad\forall\,\,X\in\XX^1(\HS).$$We denote by $\div\ss$
the $\HS$-divergence, i.e. if $X\in\XX^1(\HS)$ and $x\in
{S}$, then
$$\div\ss X (x) := \mathrm{Trace}\big(Y\longrightarrow
\gs_Y X \big)(x)\quad\,(Y\in \HH_xS).$$ The {\rm $\HS$-Laplacian}
$\Delta\ss$ is the 2nd order differential operator defined as
\[\Delta\ss\psi := \div\ss(\qq\psi)\quad\mbox{for every}\,\,\psi\in
\cont^2(S).\]The horizontal 2{nd}
fundamental form of ${S}\setminus C_S$ is the map given by
\begin{eqnarray*}{B\cc}(X,Y):=\left\langle\gc_X Y, \nn\right\rangle \qquad\forall\,\, X, Y\in\XX^1(\HS).\end{eqnarray*}
The horizontal mean curvature $\MS$ is the trace of ${{B}\cc}$,
i.e. $\MS:={\rm
Tr}B\cc=-\div\cc\nn$.\end{Defi}It is worth observing that the $\HS$-connection admits, in general, a non-zero torsion because $B\cc$ is {\it not symmetric}; see \cite{22}.

\begin{Defi}
 Let $\UU\subseteq S$ be an open set. We shall denote by $\cont^i\ss(\UU),\,(i=1, 2)$ the space of functions whose
$\HS$-derivatives up to $i$-th order are continuous on $\UU$.
\end{Defi}

We stress that the previous definitions concerning the horizontal 2nd fundamental form $B\cc(\cdot,\cdot)$ and the $\HS$-connection can also be reformulated by using the function space $\cont^i\ss(\UU),\,(i=1, 2)$ and, more precisely, by replacing $\XX^1(\HS) =\cont^1(S,\HS)$ with $\XX\ss^1(\HS):=\cont\ss^1(S,\HS)$.\\

Let $S\subset\GG$ be a hypersurface of class $\cont^i\,\,(i\geq 1)$  and let $\nu$ be the outward-pointing unit normal vector field along $S$. We need to define some important geometric objects. To this end, we first note that
 $\nu=\P\cc\nu+\P\vv\nu$. By using the left-invariant frame $\underline{X}=\{X_1,...,X_{n}\}$,  we see that $\P\vv\nu=\sum_{\alpha\in I\vv} \nu_\alpha X_\alpha$, where $\nu_\alpha:=\langle\nu, X_\alpha\rangle$; see  Notation \ref{1notlne1}.\begin{no}\label{notaziaombaretc} 
Hereafter we shall set
\begin{itemize}\item $\varpi_\alpha:=\frac{\nu_\alpha}{|\PH\nu|}\qquad\forall\,\,\alpha\in
I\vv$;\item $\varpi:=\sum_{\alpha\in I\vv}\varpi_\alpha
X_\alpha$;\item  $C\cc:=\sum_{\alpha\in {I\cd}}
\varpi_\alpha\,C^\alpha\cc$;\end{itemize}
see, for instance, Notation \ref{noconstmatr} and Remark \ref{ossconst}.\end{no}

\subsection{Other tools}\label{othy}Let $S\subset\GG$ be a hypersurface  of class $\cont^i\,\,(i\geq 1)$. Let $\partial S$ be a $(n-2)$-dimensional submanifold of $S$ of class  $\mathbf{C}^1$, oriented by the outward pointing unit normal vector
$\eta\in\TS\cap{\rm Nor}(\partial S)$.
 We shall denote by $\sigma^{n-2}\rr$ the Riemannian measure on $\partial S$,
i.e. $\sigma^{n-2}\rr\res{\partial S}=(\eta\LL\sigma^{n-1}\rr)|_{\partial S}$. In particular, note that $(X\LL\per)|_{\partial S}=\langle X, \eta\rangle
  |\PH\nu|\, \sigma^{n-2}\rr\res{\partial S}$ for every $X\in\XX^1(\TS):=\cont^1(S,\TS)$. The {\it unit
$\HS$-normal} along $\partial S$  is given by
$\eta\ss:=\frac{\P\ss\eta}{|\P\ss\eta|}$. In this way, we can define a  homogeneous $(n-2)$-dimensional measure
$\nis\in\bigwedge^{n-2}(\TT^\ast\partial S)$ by setting
  ${\nis}\res{\partial S}:=
 \left(\eta\ss\LL\per\right)\big|_{\partial S}.$
It follows that $${\nis}\res{\partial S}=
|\PH\nu|\,|\P\ss\eta|\,\sigma^{n-2}\rr\res{\partial
S}$$and that
$(X\LL\per)|_{\partial S}=\langle X, \eta\ss\rangle\,
  {\nis}\res{\partial S}$ for every $X\in\XX^1(\HS):=\cont^1(S,\HS)$.

 Now let $\nu\wedge\eta\in\Lambda^2(\TS)$ be a unit $2$-vector orienting $\partial S$, where $\nu\in{\rm Nor}(S)$ and $\eta\in \TS\cap{\rm Nor}(\partial S)$. Then, the {\it
characteristic set} of ${\partial S}$ is defined as
$C_{\partial S}:=\{p\in{\partial S}: |\P\cc(\nu\wedge\eta)|=0\}$, where the orthogonal projection operator $\P\cc$ is extended to $2$-vectors in the standard way.
 
\begin{Prop}\label{1coar1} Let
$S\subset\GG$ be a compact hypersurface of class $\mathbf{C}^1$ and let
$\phi\in\cont^1\ss(S)$. Then
\begin{eqnarray}\label{1coar}\int_{S}|\qq\phi(x)|\,\per(x)=\int_{\R}\nis\{\phi^{-1}[s]\cap
S \}ds.\end{eqnarray}
\end{Prop}

\begin{proof}
This formula follows from the Riemannian Coarea Formula; see \cite{b}, \cite{Ch4} or \cite{25}. 
We have 
\begin{equation*}\int_{S}\phi(x)\,|
\grad\ts\varphi(x)|\,\sigma^{n-1}\rr(x)=\int_{\R}ds
\int_{\varphi^{-1}[s]\cap
S}\phi(y)\,\sigma^{n-2}\rr(y)\end{equation*} for every $\phi\in
L^1(S,\sigma^{n-1}\rr)$; see
 \cite{b},  \cite{Ch4}. Choosing
 $\phi= \frac{|\grad\ss\varphi|}{|\grad\ts\varphi|}\,|\P\cc\nu|$,  yields
\begin{eqnarray*}\int_{S}\phi\,|
\grad\ts\varphi|\,\sigma^{n-1}\rr=\int_{S} \frac{|\grad\ss\varphi|}{|\grad\ts\varphi|}
|\grad\ts\varphi|\underbrace{|\P\cc\nu|\,\sigma^{n-1}\rr}_{=\per}=\int_{S} |\qq\varphi|\,\per.
\end{eqnarray*}The (Riemannian) unit normal $\eta$ along $\varphi^{-1}[s]$ is given by
$\eta=\frac{\grad\ts\varphi}{|\grad\ts\varphi|}$. Hence
$|\P\ss\eta|=\frac{|\grad\ss\varphi|}{|\grad\ts\varphi|}$ and it
turns out that
\begin{eqnarray*}\int_{\R}ds
\int_{\varphi^{-1}[s]\cap S}\phi(y)\,\sigma^{n-2}\rr&=&\int_{\R}ds
\int_{\varphi^{-1}[s]\cap
S} \frac{|\grad\ss\varphi|}{|\grad\ts\varphi|}|\P\cc\nu|\,\sigma^{n-2}\rr\\&=&\int_{\R}ds\int_{\varphi^{-1}[s]\cap
S} \underbrace{|\P\ss\eta||\P\cc\nu|\,\sigma^{n-2}\rr}_{=\nis}\\&=&\int_{\R}ds\int_{\varphi^{-1}[s]\cap
S} \,\nis.\end{eqnarray*}
\end{proof}

Below, we recall a basic integration by parts formula for horizontal vector fields; see \cite{22}.

\begin{Defi}\label{Deflh}Let $\lh:\XX\ss^1(\HS)\longrightarrow\cont(S)$  be
the 1st order
differential operator given by
\begin{eqnarray*}\lh X:=\div\ss X +  \langle C\cc\nn,
X\rangle \qquad
\forall\,\,X\in\XX\ss^1(\HS)\left( :=\cont^1\ss(S, \HS)\right) .\end{eqnarray*}Furthermore, let $\lll:\cont\ss^2(S)\longrightarrow\cont(S)$ be the 2nd order differential operator given by
\begin{eqnarray*}\lll\varphi:=\Delta\ss\varphi +  \langle C\cc\nn,
\qq\varphi\rangle
\qquad\forall\,\,\varphi\in\cont^2\ss(S);\end{eqnarray*}see Definition \ref{defiophs}  and  Notation \ref{notaziaombaretc}.
\end{Defi}
The horizontal matrix $C\cc$ is a key object,  related with the
 skew-symmetric part of the horizontal 2nd fundamental form $B\cc$. Note that  $\lh(\varphi X)=\varphi\lh X +
\langle\qq \varphi, X\rangle$   for every $X\in\XX\ss^1(\HS)$ and every $\varphi\in\cont\ss^1(S)$. Moreover, one has   $\lll \varphi=\lh(\qq
\varphi)$ for every $\varphi\in\cont\ss^2(S)$. These definitions are motivated by
Theorem 3.17, Corollary 3.18 and Corollary 3.19 in \cite{22}.

\begin{teo}[see \cite{22}]\label{GD}Let $S$ be a
compact NC hypersurface of class $\cont^2$ with
 boundary $\partial S$ of class $\cont^1$. Then
\begin{eqnarray}\label{hparts} \int_{S}\lh X\,\per=-\int_{S}\MS\langle X, \nn\rangle\,\per +
\int_{\partial S}\langle
X,\eta\ss\rangle\,\nis\qquad \forall\,\,X\in\XX^1(\HH).\end{eqnarray}\end{teo}
 \begin{oss}\label{osshsomm}We note that, in general, $\MS\notin L^1_{loc}(S; \sigma\rr^{n-1})$; see \cite{gar2}. However, it is always true that $\MS\in L^1_{loc}(S; \per)$; see, for instance, \cite{27}. 
\end{oss}

\begin{oss}
Let $S\subset\GG$ be a hypersurface of class $\cont^2$ and $\nu$ the outward-pointing unit normal vector along $S$. Let For any $X\in\XG$ let us set $X\ope:=\langle X,\nu\rangle\nu$ and $X\ot:=X-X\op$ to denote the  Riemannian normal and tangential components of $X$ at any point of $S$. We would like to stress that formula \eqref{hparts} can be seen as a particular case of a  general integral formula, the so-called 1st variation formula of the $\HH$-perimeter. More precisely, the 1st variation formula is given by\begin{equation}\label{fva}I_S(X,\per)=\int_{S}\left(-\MS\langle X\op,\nu\rangle +\div\ts\left( X\ot|\PH\nu|-\langle X\op,\nu\rangle\nn\ot \right)\right)\,\sigma\rr^{n-1}\end{equation} where $I_S(X,\per)$ denotes the 1st derivative of the $\HH$-perimeter under a smooth variation of $S$ with initial velocity $X$; see Theorem 4.6 in \cite{27}. Formula \eqref{fva} also holds   if $C_S\neq \emptyset$, but in this case we need to assume $\MS\in L^1_{loc}(S;\sigma\rr^{n-1})$.  We  observe that, in the case of the 1st Heisenberg group $\mathbb H^1$, this formula coincides with that of Ritor\'{e} and Rosales; see  \cite{32}, Lemma 4.3, p. 14.  Note that, if  $X=X\cc\in\XH$, then
\begin{eqnarray*}X\cc\ot|\PH\nu|-\langle X\cc\op,\nu\rangle\nn\ot&=&\left(X\cc-|\PH\nu|\langle X\cc,\nn\rangle\nu\right) |\PH\nu|-|\PH\nu|\langle X\cc ,\nu\rangle\left(\nn-|\PH\nu|\nu\right)\\&=&\left(X\cc-\langle X\cc ,\nu\rangle\nn\right)|\PH\nu|\\&=& \P\ss(X\cc)\,|\PH\nu|,\end{eqnarray*}where we have used the fact that $\nu=|\PH\nu|\nn+\sum_{\alpha\in I\vv}\nu_\alpha X_\alpha$ at each NC point. Finally,  inserting this  into  \eqref{fva}, we  obtain  an equivalent form of \eqref{hparts}.  In particular, for any $X\in\XH$ the function $\lh X$ turns out to be the Lie derivative of the differential  $(n-1)$-form $\per\res S$ with respect to the  initial velocity $X$ of a  smooth variation of $S$. Roughly speaking, this can be rephrased by saying that the differential  $(n-1)$-form  $(\lh X)\,\per \in\Lambda^{n-1}(\TT^\ast S)$ is the \textquotedblleft infinitesimal\textquotedblright 1st variation of $S$.
\end{oss}

Formula \eqref{hparts} holds true even if $C_S\neq \emptyset$, at least under suitable assumptions.
\begin{Defi}\label{adm}
Let $X\in\cont^1(S\setminus C_S, \HS)$ and set
 $\alpha_X:=(X\LL \per)|_S$. We say that $X$ is \rm admissible (for the horizontal divergence formula)  \it if the differential forms $\alpha_X$ and $d\alpha_X$ are continuous on all of $S$, or, more generally, if $\alpha,\,d\alpha\in L^\infty(S)$ and $\imath_S^\ast\alpha\in L^\infty(\partial S)$. We say that $\phi\in\cont^2\ss(S\setminus C_S)$  is \rm admissible \it if $\qq \phi$ is admissible for the horizontal divergence formula.
\end{Defi}
 We stress that, if the differential forms $\alpha_X$ and $d\alpha_X$ are continuous on all of $S$ (or, more generally, if $\alpha,\,d\alpha\in L^\infty(S)$ and $\imath_S^\ast\alpha\in L^\infty(\partial S)$, where $\imath_S:\partial M\longrightarrow\overline{M}$ is the natural inclusion), then Stokes formula holds true; see, for instance, \cite{Taylor}. This fact motivates the following:
\begin{corollario}\label{GDc}Let $S\subset\GG$ be a 
compact hypersurface of class $\cont^2$ with
boundary $\partial S$ of class $\cont^1$. Then\begin{itemize}
 \item[{\rm (i)}]$\int_{S}\lh X\,\per=
\int_{\partial S}\langle
X,\eta\ss\rangle\,\nis$ for every  admissible   $X\in\cont^1(S\setminus C_S, \HS)$;
 \item[{\rm (ii)}] $\int_{S}\lll \phi\,\per=
\int_{\partial S}\langle
\qq\phi,\eta\ss\rangle\,\nis$ for every  admissible   $\phi \in\cont^2\ss(S\setminus C_S)$; \item[{\rm (iii)}] if $\partial S=\emptyset$, then $-\int_{S}\varphi\lll
\varphi\,\per=\int_{S}|\qq\varphi|^2\,\per$ for every  $\varphi\in\cont^2\ss(S\setminus C_S)$ such that $\varphi^2$ is admissible. \end{itemize}
\end{corollario}The last formula holds true even if $\partial S\neq\emptyset$, but for  compactly supported functions. Moreover, it can be shown that   $\varphi^2$ is admissible if and only if $\varphi\in\cont^2\ss(S\setminus C_S)\cap W\ss^{1,2}(S,\per)$ where we have set $W\ss^{1,2}(S,\per):=\{\varphi\in L^2(S,\per): |\qq\varphi|\in L^2(S,\per)\}$.
We also remark that any vector field  $X\in\cont^1(S,\HS)$ turns out to be admissible. Analogously, any  $\varphi\in\cont^2\ss(S)$ is admissible.

\begin{lemma}\label{calcolodivx}Let  $\x\cc:=\sum_{i\in I\cc}x_i X_i$ be the
{``horizontal position vector''} and let $\g$ denote its
component along the $\HH$-normal $\nn$, i.e. $\g:=\langle
\x\cc,\nn\rangle$. In the sequel, the function $\g$ will be called
``horizontal support function'' of $\x\cc$.  Then, we have:\begin{itemize}\item[$\rm (i)$] $\div\cc x\cc=\DH$;
\item[$\rm (ii)$] $\lh(x\ss)=(\DH-1)+ \g \MS +  \langle C\cc\nn,\x\ss\rangle$ at each NC point $x\in S\setminus C_S$, where $x\ss:=x\cc-\g\nn$.\end{itemize}\end{lemma}

\begin{proof}We have $\div\cc x\cc=\sum_{i=1}^\DH\langle\nabla_{X_i}x\cc,X_i\rangle=\sum_{i, j=1}^\DH\left( X_i(x_j)+ \langle\nabla_{X_i}X_j,X_i\rangle\right)=\sum_{i, j=1}^\DH \delta_i^j=h$, where $\delta_i^j$ denotes Kronecker's delta. Note that we have used $\mathcal J\cc(x\cc)={\bf Id}_h$ and $\langle\nabla_{X_i}X_j,X_i\rangle=0$ for all $i, j\in I\cc$; see Definition \ref{graddivjac} and formula \eqref{graddivjac}.
Furthermore, by definition, one has $\div\ss x\cc=\div\cc x\cc-\left\langle\nabla_{\nn} x\cc,\nn\right\rangle$. Hence 
$\div\ss x\cc=\DH-\left\langle \nn,\nn\right\rangle=\DH-1$. Furthermore, by definition, we have
\begin{equation}\label{mkmk1}\div\ss x\ss=\sum_{i=2}^\DH\langle\nabla_{\tau_i}\left(x\cc-\g\nn\right),\tau_i\rangle,\end{equation}where we have used an orthonormal  horizontal frame $\underline{\tau}\,\cc:=\left\lbrace\tau_1,...,\tau_h\right\rbrace$ in an open  neighborhood $U\subset\GG$ of $S$ such that $\tau_1(x)=\nn(x)$ at any $x\in S\setminus C_S$; see, for instance, Definition 3.4 in \cite{22}. Starting from \eqref{mkmk1}, we compute $$\div\ss x\ss=\sum_{i=2}^\DH\left(\langle\tau_i,\tau_i\rangle-\g\langle\gc_{\tau_i}\nn,\tau_{_i}\rangle \right)=(\DH-1)-\g\divh\nn=(\DH-1)+\g\MS$$for every  $x\in S\setminus C_S$. The thesis easily  follows from the definition of $\lh$.  
\end{proof}

A simple consequence of Corollary \ref{GDc} and Lemma  \ref{calcolodivx}. is given by the following:
\begin{corollario}[Minkowsky-type formula]\label{Mink}Let $S\subset\GG$ be a  
compact  hypersurface of class $\cont^2$. Let  $\x\cc=\sum_{i\in I\cc}x_i X_i$ be the
 horizontal position vector. Furthermore, set $\g=\langle
\x\cc,\nn\rangle$ and $\x\ss=\x-\g\nn$ for every $x\in S\setminus C_S$. Then
\[ \int_{S} \left(  (\DH-1)+ \g\MS +  \langle C\cc\nn,
\x\ss\rangle\right)  \per = 0.\]
\end{corollario}
\begin{proof}It is enough to apply  Corollary \ref{GDc} to the horizontal tangent vector field $\x\ss\in\cont^1(S\setminus C_S,\HS)$. Using Remark \ref{osshsomm} and Lemma  \ref{calcolodivx} the thesis easily follows. \end{proof}

\begin{Defi}[Eigenvalue problems for $\lll$]\label{epr} Let $S\subset\GG$ be a compact hypersurface of class $\cont^2$
without boundary.  Then we look for solutions of class $\cont^2\ss(S\setminus C_S)\cap W\ss^{1,2}(S,\per)$ to
the problem:
\begin{displaymath}
{\rm(P_1)}\,\,\,\left\{%
\begin{array}{ll}
  \,\,\,\,\,-\lll\psi=\lambda \, \psi; \\
  \int_S\psi\,\per  =0. \\
\end{array}%
\right.\end{displaymath}If $\partial S\neq {\emptyset}$, we look
for solutions of class  $\cont^2\ss(S\setminus C_S)\cap W\ss^{1,2}(S,\per)$ to
the problems:
\begin{displaymath}
{\rm(P_2)}\,\,\,\left\{%
\begin{array}{ll}
  \, -\lll\psi=\lambda \, \psi; \\
   \quad \,\,\psi|_{\partial S} =0; \\
\end{array}%
\right.\qquad
{\rm(P_3)}\,\,\,\left\{%
\begin{array}{ll}
    \,\,-\lll\psi=\lambda \,\psi; \\
    \,\frac{\partial\psi}{\partial\eta\ss}\big|_{\partial S} =0.
\end{array}%
\right.\end{displaymath}We explicitly remark that
$\frac{\partial\psi}{\partial\eta\ss}=\langle\qq\psi,\eta\ss\rangle$.

\end{Defi}

The problems (P$_1$), (P$_2$) and (P$_3$) generalize to our
context the classical {\it closed, Dirichlet and Neumann
eigenvalue problems} for the Laplace-Beltrami operator on
Riemannian manifolds; see \cite{Ch1, Ch3}.

Finally, we recall a recent general
result about the
size of horizontal tangencies to non-involutive distributions, which applies to our Carnot setting; see Theorem 4.5 in
\cite{Bal3}.
\begin{teo}[Generalized Derridj's Theorem]\label{baloghteo}Let $\GG$ be a $k$-step Carnot group.

\begin{itemize}\item[{\rm (i)}]If $S\subset \GG$ is a hypersurface
of class $\cont^2$,  the Euclidean-Hausdorff dimension of the
characteristic set $C_S$ of $S$ satisfies $\dim_{\rm
Eu-Hau}(C_N)\leq n-2.$\item[{\rm (ii)}]If
$\VV=\HH^\perp\subset\TG$ satisfies $\dim \VV\geq 2$ and $N\subset
\GG$ is a $(n-2)$-dimensional submanifold of class $\cont^2$, then
the Euclidean-Hausdorff dimension of the characteristic set $C_N$
of $N$ satisfies $\dim_{\rm Eu-Hau}(C_N)\leq n-3.$\end{itemize}
\end{teo}

\begin{oss}\label{monteoss}Let
$N\subset \GG$ be a $(n-2)$-dimensional submanifold of class
$\cont^2$. This smoothness condition  is sharp, see \cite{Bal3}.
Moreover, we stress that $\dim \VV=1$ just for  Heisenberg groups  and $2$-step Carnot groups  having $1$-dimensional center.
For Heisenberg groups $\mathbb{H}^n,\,n>1$,  using
 Frobenius' Theorem yields
$\dim_{\rm Eu-Hau}(C_N)\leq n$, where $n=\frac{\dim\HH}{2}$; see also \cite{Bal3}. On
the contrary,
$1$-dimensional curves in $\mathbb{H}^1$, can be  horizontal or transversal to
$\HH$. For $2$-step groups having $1$-dimensional center (or, equivalently,  horizontal bundle $\HH$ of codimension $1$)  a simple analysis shows that $\dim_{\rm
Eu-Hau}(C_N)= n-2$ if, and only if, $\GG$ reduces to the  direct product of $\mathbb{H}^1$  and of a Euclidean space $\R^{\DH-2}$.\end{oss}

\section{Isoperimetric constants and the 1st eigenvalue  of $\lll$ on compact hypersurfaces}\label{poincinsect}

As a consequence of the Coarea Formula \eqref{1coar} we may
generalize to the Carnot groups setting some results about
isoperimetric constants and global Poincar\'e inequalities for which
we refer the reader to \cite{Ch1, Ch3}; see also
\cite{Cheeger}, \cite{Yau}.

Let $S\subset\GG$ be a compact hypersurface of class $\cont^2$ with
(or without) boundary. Similarly as in the Riemannian setting (see
\cite{Cheeger} and \cite{Yau}), we may give the following: \begin{Defi}\label{D1}The {\rm
isoperimetric constant} ${\rm Isop}(S)$ of $S$ is defined as
follows:\begin{itemize} \item if $\partial S=\emptyset$, we set
$${\rm Isop}(S):=\inf\frac{\nis(N)}{\min\{\per(S_1),\per(S_2)\}},$$where
the infimum is taken over all $\cont^2$-smooth $(n-2)$-dimensional
submanifolds $N$ of $S$ which divide $S$ into two hypersurfaces
$S_1, S_2$ with common boundary $N=\partial S_1=\partial S_2$;
\item if $\partial S\neq \emptyset$, we set
$${\rm Isop}(S):=\inf\frac{\nis(N)}{\per(S_1)},$$where
$N\subset S$ is a smooth hypersurface of $S$ such that $N\cap \partial S=\emptyset$ and $S_1$ is the unique $\cont^2$-smooth
$(n-2)$-dimensional submanifold of $S$ such that  $N=\partial
S_1$.

\end{itemize}
Here above $\partial S,\, S_1,\, S_2 $ and $N=\partial S_i\,(i=1,2)$ are not assumed to be connected.
\end{Defi}
This definition requires some comments. As recalled in the
introduction, in the Riemannian setting analogous isoperimetric
constants were introduced by Cheeger in \cite{Cheeger}, in order
to give a geometric lower bound for the smallest eigenvalue of the
Laplace-Beltrami operator on smooth compact Riemannian manifolds.
This definition was somewhat  motivated by an example of Calabi,
the so-called \it dumbbell \rm manifold, homeomorphic to ${\bf
S}^2$. Actually, an analysis of this example shows that, in order
to bound $\lambda$ from below, the diameter and the volume are not
enough.

We also have to recall that these isoperimetric constants turn out to be strictly positive. Although, this claim turns out to be (more or less) elementary in dimension $n=2$, it becomes a bit more difficult when $n>2$; see \cite{Cheeger}.
Some years later after Cheeger result, Yau (see \cite{Yau}) reconsidered the isoperimetric constants and demonstrated that $\lambda$ has a bound in terms of volume, diameter and (of a lower bound of the) Ricci curvature. See the survey \cite{Li} for a glimpse on this topic.

 Below we shall generalize some of the results of  \cite{Yau}. Our results will follow the original scheme,   which is based mainly on a suitable use of the Coarea formula for smooth functions. Note also that, instead of $\cin$-smooth hypersurfaces, here we are considering hypersurfaces of class $\cont^2$. We have to observe that all the results could also be stated for  $\cont^1$ hypersurfaces. But the delicate matter here is that in our   setting, new difficulties come from the presence of characteristic points and, in the $\cont^1$ case, it is not  simple to prove that
isoperimetric constants are strictly positive. Actually, the following further hypothesis seems to be unavoidable in order to have  non-zero isoperimetric constants:
\begin{itemize}\item[${\bf (H)}$]\it every $\cont^2$-smooth
$(n-2)$-dimensional submanifold $N\subset S$ satisfies $\dim\,
C_{N}<n-2.$
\end{itemize}
This  assumption can be overcome by using the generalized Derridj's Theorem \ref{baloghteo}; see also Remark \ref{monteoss}. As a consequence, the results of this section  are \textquotedblleft meaningful\textquotedblright (in the sense that the isoperimetric constants do not vanish) at least for  any Carnot group $\GG$ such that $\dim \VV\geq 2$ and for all Heisenberg groups $\mathbb H^n$, with $n>1$.

\begin{teo}\label{zaz}Let $S\subset\GG$ be a compact hypersurface of class $\cont^2$. \begin{itemize} \item[{\rm (i)}]
If $\partial S=\emptyset$, then
$${\rm Isop}(S)=\inf\frac{\int_S|\qq \psi|\,\per}{\int_S|\psi|\,\per},$$where
the infimum is taken over all $\cont^2$-smooth functions on $S$
such that $\int_S\psi\,\per=0$. \item[{\rm (ii)}] If $\partial S\neq
\emptyset$, then $${\rm Isop}(S)=\inf\frac{\int_S|\qq
\psi|\,\per}{\int_S|\psi|\,\per},$$ where the infimum is taken over
all $\cont^2$-smooth functions on $S$ such that $\psi|_{\partial
S}=0$.
\end{itemize}
\end{teo}

\begin{war}\label{os1}The definition of ${\rm Isop}(S)$  can be weakened.
For instance, $(i)$ of Definition \ref{D1} can be given by
assuming $S$ of class $\cont^1$ and then by taking the infimum
over  all $(n-2)$-dimensional submanifolds $N$ of $S$ of class
$\cont^1$ which divide $S$ into two hypersurfaces $S_1, S_2$ with
common boundary $N=\partial S_1=\partial S_2$. In this case,
$(i)$ of  Theorem \ref{zaz} holds, without modifications,
  by taking the infimum over $\cont\ss^1$-smooth functions. If $\partial S\neq \emptyset$ an analogous claim holds, for the other isoperimetric constant. Furthermore, equivalent remarks can be given for all the results of this section.
Nevertheless, as already said, this weaker formulation seems to be less meaningful because of the presence of characteristic points.
\end{war}

\begin{war}Throughout this section, we shall fix a homogeneous distance $\varrho$ on $\GG$ of class $\cont^1$ outside the diagonal of $\GG$.\end{war}

\begin{proof}[Proof of Theorem \ref{zaz}]The proof repeats almost verbatim the arguments of Theorem 1 in
\cite{Yau}. We just prove the theorem for $\partial S=\emptyset$
since the other case is analogous. First, let us prove the
inequality

$${\rm Isop}(S)\leq\inf\frac{\int_S|\qq
\psi|\,\per}{\int_S|\psi|\,\per}$$where $\psi\in\cont^2(S)$ and
$\int_S\psi\,\per=0$. To prove this inequality let us consider the
auxiliary functions $\psi^{+}=\max\{0, \psi\},\, \psi^{-}=\max\{0,
-\psi\}$. By applying the Coarea Formula \eqref{1coar} and the 
definition of ${\rm Isop}(S)$ we get that$$\int_{S}|\qq\psi^{\pm}|\,\per=\int_0^{+\infty}\nis\{x\in
S: \psi^\pm=t\}\,dt\geq{\rm Isop}(S)\int_S|\psi^\pm|\,\per.$$Now we
shall prove the reversed inequality. So let us assume that
$\per(S_1)\leq\per(S_2)$ and let $\epsilon>0$. By making use of the fixed homogeneous distance $\varrho$ on $\GG$, we now define a
 function $ \psi_{\epsilon}:S\longrightarrow\R$ by setting
\begin{eqnarray}\label{ODER1}   \psi_{\epsilon}(x)|_{S_1}:=
\left\{\begin{array}{ll} \frac{\varrho(x, N)}{\epsilon}
\,\,\,\mbox{if }\,\, \varrho(x, N)\leq \epsilon
\\\\
\,\,\,\,\,\,\, 1\quad\mbox{if } \, \varrho(x, N)> \epsilon\end{array}
,\quad\right.\quad\psi_{\epsilon}(x)|_{S_2}:= \left\{\begin{array}{ll}
-\alpha\frac{\varrho(x, N)}{\epsilon} \,\,\,\mbox{if }\,\,
\varrho(x, N)\leq \epsilon
\\\\
\,\,\,\,  -\alpha  \quad\quad\mbox{if } \,\,\,\varrho(x, N)>
\epsilon\end{array},\right.\end{eqnarray}where the constant
$\alpha$ depends on $\epsilon$ and is chosen in a way that
$\int_S\psi_\epsilon\,\per=0.$ Obviously
$$\lim_{\epsilon\rightarrow
0}\alpha=\frac{\per(S_1)}{\per(S_2)}.$$ Since
\begin{eqnarray*}\int_{S}|\qq\psi_\epsilon|\,\per&=&\frac{1+\alpha}
{\epsilon}\int_{N_\epsilon:=\{x\in S:\varrho(x,
N)\leq\epsilon\}}|\qq\varrho(x,N)|\,\per\\&=&\frac{1+\alpha}
{\epsilon}\int_{0}^{\epsilon}\nis\{x\in N_\epsilon: \varrho(x,
N)=t\}\,dt,\end{eqnarray*}one gets $$\lim_{\epsilon\rightarrow
0}\int_{S}|\qq\psi_\epsilon|\,\per=(1+\alpha)\,\nis(N).$$Moreover
$\lim_{\epsilon\rightarrow
0}\int_{S}|\psi_\epsilon|\,\per=\per(S_1)+\alpha\,\per(S_2)$. Putting
all together we  get$$\inf_{\psi}\frac{\int_S|\qq
\psi|\,\per}{\int_S|\psi|\,\per}\leq\lim_{\epsilon\rightarrow
0}\frac{\int_S|\qq
\psi_\epsilon|\,\per}{\int_S|\psi_\epsilon|\,\per}\leq  \frac{\per(N)}{\nis(S_1)}.
$$If we take the infimum over $N$ and $S_1$, the inequality follows.
\end{proof}

\begin{corollario}\label{555}Let $\lambda_1$ be the first non-zero eigenvalue of either the closed eigenvalue problem
 or the Dirichlet  eigenvalue problem; see Definition \ref{epr}.
  Then $\lambda_1\geq \frac{({\rm Isop}(S))^2}{4} $. \end{corollario}

\begin{proof}We just prove the  first claim, as the second claim is similar.  Let $\psi$ be an eigenfunction of $\lll$ corresponding to $\lambda_1$. Then
\begin{eqnarray*}\lambda_1&=& -\frac{\int_S\psi\,\lll\psi\,\per}{\int_S|\psi|^2\per}= \frac{\int_S|\qq\psi|^2\per}{\int_S|\psi|^2\per}\\&=&   \frac{\int_S|\qq\psi|^2\per}{\left( \int_S|\psi|^2\per\right)^2} \int_S|\psi|^2\per \\&\geq&   \frac{\left(\int_S|\psi||\qq\psi|\,\per\right) ^2}{\left( \int_S|\psi|^2\per\right)^2} \\&=&   \frac{1}{4}\frac{\left(\int_S |\qq\psi^2|\,\per\right)^2}{\left( \int_S \psi^2\per\right)^2}\geq \frac{\left({\rm Isop}(S)\right)^2}{4},
\end{eqnarray*}where we have used Theorem \ref{zaz} together with Cauchy-Schwarz inequality.
\end{proof}

We now extend, to Carnot groups, another isoperimetric
constant  and some related facts which, in the Riemannian case, were studied  in \cite{Yau}.
\begin{Defi} The {\it isoperimetric constant}
${\rm Isop}_0(S)$ of any $\cont^2$-smooth  compact hypersurface
 $S\subset\GG$ with boundary $\partial S$ is given by
$${\rm Isop}_0(S):=\inf\left\{\frac{\nis(\partial S_1\cap \partial S_2)}{\min\{\per(S_1),\per(S_2)\}}\right\},$$where
the infimum is taken over all decompositions $S=S_1\cup S_2$ such
that $\per (S_1\cap S_2)=0$.

\end{Defi}

\begin{teo}\label{zdaz}Let $S\subset\GG$ be a compact hypersurface of class $\cont^2$ with boundary.
Then
$${\rm Isop}_0(S)=\inf\left\{\frac{\int_S|\qq \psi|\,\per}{\inf_{\beta\in \R}\int_S|\psi-\beta|\,\per}\right\},$$where
the $\inf$ is taken over all $\cont^2$-functions defined on $S$.
\end{teo}
\begin{proof}The proof is analogous to that of Theorem 6 in
\cite{Yau}. First, let us prove the inequality

$${\rm Isop}(S)\leq\inf\frac{\int_S|\qq
\psi|\,\per}{\int_S|\psi|\,\per}.$$To this purpose, let us define the
functions $\psi^{+}:=\max\{0, \psi-k\},\, \psi^{-}:=-\min\{0,
\psi-k\}$, where $k\in\R$ is any constant such that:\begin{eqnarray*}\per\{x\in S: \psi^{+}>0\}&\leq& \frac{1}{2}\per(S),\\
\per\{x\in S: \psi^{-}>0\}&\leq
&\frac{1}{2}\per(S).\end{eqnarray*}By using again the Coarea
Formula \eqref{1coar} together with the  definition of ${\rm Isop}_0(S)$
we get
that$$\int_{S}|\qq\psi^{\pm}|\,\per=\int_0^{+\infty}\nis\{x\in S:
\psi^\pm=t\}\,dt\geq{\rm Isop}(S)\int_S|\psi^\pm|\,\per.$$We
prove the other inequality. Assuming $\per(S_1)\leq\per(S_2)$ and
$\epsilon>0$, we define  the
 function

\begin{eqnarray}\label{ODER2}\psi_{\epsilon}(x)|_{S_1}:=1,\qquad\psi_{\epsilon}(x)|_{S_2}:=
 \left\{\begin{array}{ll}
1-\frac{\varrho(x,\partial S_1\cap
\partial S_2)}{\epsilon}\,\,\,\mbox{if }\,\, \varrho(x,\partial S_1\cap
\partial S_2))\leq
\epsilon
\\\\
\,\,\,\qquad   0 \,\qquad\qquad\mbox{if } \,\,\varrho(x,\partial
S_1\cap
\partial S_2))>
\epsilon.\end{array} \right.\end{eqnarray}Furthermore, one
can find a constant $k(\epsilon)$ satisfying
$$\int_{S}|\psi_\epsilon-k(\epsilon)|\,\per=\inf_{\beta\in\R}\int_S|\psi_\epsilon-\beta|\,\per$$and
such that $k(\epsilon)\longrightarrow 0$ for $\epsilon\rightarrow
0^+$. Hence $$\lim_{\epsilon\rightarrow
0}\left\{\frac{\int_S|\qq \psi_\epsilon|\,\per}{\inf_{\beta\in
\R}\int_S|\psi_\epsilon-\beta|\,\per}\right\}\leq\frac{\nis(\partial
S_1\cap
\partial S_2)}{\min\{\per(S_1),\per(S_2)\}}.$$

\end{proof}

\begin{corollario} \label{1zdaz}Let $S\subset\GG$ be a  compact hypersurface of class $\cont^2$. Then
\begin{eqnarray}\label{gay}\int_S|\psi-k|^2\per\leq \frac{4}{\left( {\rm
Isop}_0(S)\right) ^2}\int_S|\qq\psi|^2\,\per
\end{eqnarray}for every $\psi\in\cont^2(S)$ and every $k\in\R$
such that\begin{eqnarray*}\per\{x\in S: \psi\geq k\}&\geq& \frac{1}{2}\per(S),\\
\per\{x\in S: \psi\leq k\}&\geq
&\frac{1}{2}\per(S).\end{eqnarray*}Furthermore, if
$\psi\in\cont^2(S)$ and $\int_S\psi\,\per=0$,
then\begin{eqnarray}\label{ga0}\int_S|\psi|^2\per\leq
\frac{4}{({\rm Isop}_0(S))^2}\int_S|\qq\psi|^2\per.
\end{eqnarray}
\end{corollario}

\begin{proof}One has
$\int_S(\psi^+\cdot\psi^-)\,\per=0$, where the functions
$\psi^\pm$ are defined as in the proof of Theorem
\ref{zdaz}. Moreover, by using once more Coarea Formula, we
get\begin{eqnarray*}\int_S|\psi-k|^2\per&=&\int_S|\psi^+ +\psi^-|^2\per\\&\leq&
\int_S|\psi^+|^2\per+\int_S|\psi^-|^2\per\\& \leq&\frac{1}{{\rm
Isop}_0(S)}\left(\int_S|\qq(\psi^+)^2|\,\per+\int_S|\qq(\psi^-)^2|\,\per\right)\\&\leq&
\frac{2}{{\rm Isop}_0(S)}\int_S(\psi^+
+\psi^-)|\qq\psi|\,\per\\&\leq&\frac{2}{{\rm
Isop}_0(S)}\|\psi^++\psi^-\|_{L^2(S; \per)}\|\qq\psi\|_{L^2(S;
\per)}.\end{eqnarray*}This proves  \eqref{gay}. In order to prove
\eqref{ga0} we note that the hypothesis $\int_S\psi\,\per=0$
actually implies that
\[\int_S\psi^2  \per=\inf_{k\in\R}\int_S(\psi-k)^2 \per,\]which
together with \eqref{gay}, implies the thesis of the theorem.
\end{proof}
\section{Two upper bounds on $\lambda_1$}\label{ChRe}

Below we shall extend two (nowadays classical) inequalities obtained, respectively,  by Chavel and Reilly in the Euclidean/Riemannian setting. An important feature of these results is in that they give explicit upper bounds for the first non-trivial eigenvalue (of the Laplacian) of a compact submanifold of $\Rn$. For further details  we refer to \cite{88} and \cite{222}; see also \cite{111}. To begin with, let $\Om\subsetneq\GG$ be a bounded domain and assume that $S:=\partial\Om$ is a connected hypersurface of class $\cont^2$, with orientation given by the outward normal vector $\nu$.
Moreover, let $\x\cc$ be the horizontal position vector field and let us apply the usual divergence formula.
We also set $\sigma^n\rr=\Vol$. We have 
$$\DH\,\Vol(\Om)= \int_\Om\div\cc \x\cc \,\sigma^n\rr=\int_{\partial \Om}\langle\x\cc, \nu\rangle\,\sigma^{n-1}\rr=\int_{S}\langle\x\cc, \nn\rangle\,\per,$$where we have used identity $\rm(i)$ of Lemma \ref{calcolodivx}.
Furthermore,  we may further assume that the \textquotedblleft center of mass\textquotedblright of $\partial \Om$ (with respect to the $\HH$-perimeter) is placed at the identity $0\in\GG$. In other words, let us assume that $\int_S \x_i\,\per=0$ for every $i\in I\cc=\{1,...,\DH\}$, where $\x\cc\equiv(\x_1,...,\x_i,...,\x_\DH)$ is the horizontal position vector; see Lemma \ref{calcolodivx}.

 The last assumption is justified by the following:
\begin{lemma} Let $S\subset\GG$ be a compact hypersurface of class $\cont^i\,\,(i\geq 1)$. We can always choose a system of exponential coordinates $x=\exp(x_1,...,x_n)$ on $\GG$  such that
 $\int_S x_i\,\per(x)=0$ for every $i\in I\cc=\{1,...,\DH\}$.
\end{lemma}\begin{proof}Let$$a_i:=\frac{\int_Sx_i\,\per(x)}{\per(S)}\qquad\forall\,\,i\in I\cc=\{1,...,\DH\}$$and  $a\cc\equiv (a_1,...,a_i,...,a_\DH)$.  Set $a:=\exp(a\cc,0\vv)$, where  the symbol $0\vv$ denotes the origin of  $\VV\subset \gg$. Consider the change of variables $y:=\Phi(x)= a^{-1}\bullet x\,\,(x\in\GG)$.  Equivalently, we have $\Phi(x)=L_{a^{-1}}(x)$, where $L_{a^{-1}}$ is the left-translation by ${a^{-1}}=-a$; see Section \ref{cg}. The usual Change of Variables formula together with standard properties of the pull-back imply the following chain of equalities:\begin{equation}\label{teceqy}
\int_{\Phi(S)}f(y)\,\per(y)=\int_{S}f\left(\Phi(x)\right)\,{\mathcal J}ac(\Phi)(x)\,\per(x)=\int_{S}\Phi^\ast\left(f\per\right)=\int_{S}\left(f\circ\Phi\right)\, \left(\Phi^\ast\per\right)
\end{equation} for every smooth function $f:S\longrightarrow\R$; see, for instance, Lee's book \cite{Lee} Lemma 9.11, p. 214.
Using the left-invariance of the $\HH$-perimeter yields ${\mathcal J}ac(\Phi)=1$, or equivalently, $\Phi^\ast\per=\per$.
Now, let us assume that $f(y):=y_i$ for any $i\in I\cc$. Equivalently, let $f$ be the $i$-th  exponential coordinate of the variable $y\in\GG$. Note also that $\left(f\circ\Phi\right)(x)=\Phi_i(x)=-a_i+x_i$ for any $i\in I\cc$. Actually, this follows from the fact that the group law $\bullet$ acts linearly on the horizontal layer; see \eqref{CBHformula}. Then, using \eqref{teceqy} yields
$$\int_{\Phi(S)}y_i\,\per(y)=\int_S(-a_i+x_i)\,\per(x)=0\qquad\forall\,\,i\in I\cc,$$which achieves the proof. 
\end{proof}

We therefore get that
 \begin{eqnarray*}\DH\,\Vol(\Om)&=&\int_{S}\langle\x\cc, \nn\rangle\,\per\\&\leq&\int_{S}|\x\cc|\,\per\\&\leq& \sqrt{\per(S)} \sqrt{\int_{S}|\x\cc|^2\,\per}\\&=&\sqrt{\per(S)} \sqrt{\int_{S}\sum_{i\in I\cc}\x_i^2\,\per}\\&\leq&\sqrt{\dfrac{\per(S)}{\lambda_1}} \sqrt{\int_{S}\sum_{i\in I\cc}|\qq \x_i|^2\,\per},
\end{eqnarray*}where the last identity follows from Lord Rayleigh's characterization of the 1st non-trivial eigenvalue $\lambda_1$ of the operator $\lll$ on $S$.
Now a direct computation gives the pointwise identity $\sum_{i\in I\cc}|\qq \x_i|^2=\DH-1$.
Hence, putting all together, we have shown the following:
\begin{teo} Let $\Om\subsetneq\GG$ be a bounded domain with $\cont^2$ boundary $S=\partial D$. Moreover, let $\lambda_1$ be the 1st (non-trivial) eigenvalue of the operator $\lll$ on $S$. Then
 \begin{eqnarray*} \sqrt{\lambda_1}\,\dfrac{\Vol(\Om)}{\per(S)} \leq \dfrac{\sqrt{\DH-1}}{\DH}.
\end{eqnarray*}
\end{teo}
 We now discuss another geometric inequality, which looks very similar  to the last one. More precisely, let $S$ be a $\cont^2$-smooth compact hypersurface without boundary. So let us make use of Rayleigh's principle$$\int_S\varphi^{2}\,\per\leq\int_S|\qq\varphi|^{2}\,\per$$ for any function $\varphi\in\cont^2(S\setminus C_S)\cap W^{1,2}\ss(S,\per)$ satisfying $\int_S\varphi\,\per=0$. Again, we assume that the center of mass of $S=\partial \Om$ is placed at $0\in\GG$ so that $\int_S \x_i\,\per=0$ for every $i\in I\cc$. Hence, similarly as above,  we get that
\begin{eqnarray*}
\lambda_1\int_S|\x\cc|^2\,\per=\lambda_1\sum_{i\in I\cc}\int_S\x_i^2\,\per\leq\lambda_1\sum_{i\in I\cc}\int_S|\qq\x_i|^2\,\per=(\DH-1)\,\per(S).
\end{eqnarray*}
At this point, we reformulate Corollary \ref{Mink} as follows:
\[ \int_{S}  \left(  (\DH-1)+ \left\langle\left( \MS\nn +   C\cc\nn\right),
\x\cc\right\rangle  \right)\per = 0.\]From this identity and Cauchy-Schwartz inequality, we easily get that
\begin{eqnarray*} (\DH-1)\,\per(S)&\leq& \sqrt{\int_S|\x\cc|^2\,\per} \sqrt{\int_S\left|\MS\nn +   C\cc\nn\right|^2\,\per} \\&\leq& \sqrt{\int_S|\x\cc|^2\,\per} \sqrt{\int_S\left(\MS^2 +   |C\cc\nn|^2\right)\,\per}.\end{eqnarray*}
Therefore $$\dfrac{\left((\DH-1)\,\per(S)\right)^2}{\int_S\left(\MS^2 +   |C\cc\nn|^2\right)\,\per}\leq\int_S|\x\cc|^2\,\per $$and hence\begin{eqnarray*}
\lambda_1\dfrac{\left((\DH-1)\,\per(S)\right)^2}{\int_S\left(\MS^2 +   |C\cc\nn|^2\right)\,\per}\leq(\DH-1)\,\per(S),
\end{eqnarray*}which proves the following:
\begin{teo}  Let $\Om\subsetneq\GG$ be a bounded domain with $\cont^2$ boundary $S=\partial D$ and $\nu$  the outward-pointing unit normal vector along $S$. Moreover, let $\lambda_1$ be the 1st eigenvalue of the operator $\lll$ on $S$. Then, the following upper bound for $\lambda_1$ holds
 \begin{eqnarray*}\lambda_1\leq\frac{\int_S\left(\MS^2 +|C\cc\nn|^2\right)\,\per}{(\DH-1)\,\per(S)}=\dfrac{\fint_S\left( \MS^2 +|C\cc\nn|^2\right)\,\per}{\DH-1}
\end{eqnarray*}
\end{teo}

\section{Horizontal Linear Isoperimetric inequalities}\label{wlinisoineq}
Let $S\subset\GG$ be a compact
 hypersurface of class $\cont^2$ with (or without) boundary. Let $\x\cc$ be the horizontal position vector of $S$ and set $x\ss:=\x\cc-\g\nn$ where 
$\g=\langle \x\cc,\nn\rangle$ is the horizontal support
function of $S$; see Lemma \ref{calcolodivx}. We recall that
\begin{equation}\label{jkjlkloko} \int_{S}\left((\DH-1)+ \g\MS +  \langle C\cc\nn,
\x\ss\rangle\right)\per = \int_{\partial S}\langle
\x\cc,\eta\ss\rangle\,\nis;\end{equation}see
Corollary \ref{Mink}. Note  that, if $\partial S=\emptyset$, then the boundary integral vanishes. From this we easily get that

\begin{eqnarray}\label{bla}(\DH-1)\,\per(S)\leq \int_S\left(|\g||\MS| +  |\langle C\cc\nn,
\x\ss\rangle|\right)\per + \int_{\partial S}|\langle
\x\cc,\eta\ss\rangle|\nis.\end{eqnarray}

\begin{oss}[Assumptions on $\varrho$]\label{iponhomnor}
Let $\varrho(x)=\varrho(0,x)=\|x\|_\varrho$ be a homogeneous norm on $\GG$ and let $\varrho(x, y)=\|y^{-1}\bullet x\|_\varrho$ be the associated (homogeneous) distance on $\GG$. In this section we assume the following: 
\begin{itemize}\item[{\rm(i)}] $\varrho$ is  piecewise $\cont^1$ outside the diagonal of $\GG$;\item[{\rm(ii)}]$|\grad\cc\varrho|\leq 1$ at each regular point of $\varrho$;
\item[{\rm(iii)}]${|x\cc|}\leq {\varrho(x,0)}\quad\forall\,\,x\in\GG$.\end{itemize}
\end{oss}

\begin{es}\label{Kor}On the
Heisenberg group $\mathbb{H}^n$, the CC-distance $\dc$ satisfies these assumptions. Another example is the distance associated with the  Korany norm defined as
$\|x\|_\varrho:=\varrho(x)=\sqrt[4]{|x\cc|^4+16t^2}$ for $x=\exp(x\cc,t)\in\mathbb{H}^n$. This norm is
homogeneous and
  $\cin$-smooth  out of
$0\in\mathbb{H}^n$ and
satisfies conditions \rm (ii) \it and \rm (iii). \it This example can easily be
generalized to any Carnot group having step 2 and satisfying $C^\alpha\cc C^\beta\cd=-\mathbf{1}\ci\delta_{\alpha}^{\beta},$ $(\alpha, \beta\in I\cd)$. Actually, in this case, one can show that the homogeneous norm $\|\cdot\|_\varrho$,
defined by
$\|x\|_\varrho:=\sqrt[4]{|x\cc|^4+16|x\cd|^2}\,\,\forall\,x=\exp(x\cc, x\cd),$ satisfies all
the  conditions in Remark \ref{iponhomnor}.
\end{es}
 Let  $R$ be the radius of the $\varrho$-ball $B_\varrho(0,R)$, centered at the identity $0$ of the group $\GG$ and circumscribed
about $S$. It is important to remark that, because of the left-invariance of the $\HH$-perimeter,  we may replace $0$ with any $x\in\GG$. Below,  we shall estimate (by Cauchy-Schwarz inequality) the right-hand side of
\eqref{bla}. To this aim, note that $\g\leq
|\x\cc|\leq\|x\|_\varrho$. So we have
\begin{equation}\label{1pisoplin}(\DH-1)\,\per(S)\leq
R\left(\int_S\left(|\MS|+ |C\cc\nn|\right)\per+
\nis(\partial S)\right),\end{equation}which is a linear inequality. Obviously, if $S$ is
$\HH$-minimal, i.e. $\MS=0$, it follows that\begin{equation}\label{2isoplin}(\DH-1)\,\per(S)\leq
R\left(\int_S |C\cc\nn|\,\per+
\nis(\partial S)\right).\end{equation}Furthermore, if
$\MS^0:=\max\{\MS(x)|x\in S\}$, one gets
\begin{equation}\label{3isoplin}\per(S)\left((\DH-1)-R\MS^0\right)\leq
R\left(\int_S |C\cc\nn|\,\per+
\nis(\partial S)\right).\end{equation}Equivalently, we
have\begin{equation}\label{4isoplin}R\geq\frac{(\DH-1)\,\per(S)}{\MS^0\per(S)+\left(\int_S
|C\cc\nn|\,\per+ \nis(\partial S)\right)},
\end{equation}and, by assuming $R\MS^0<\DH-1$, we also get that
\begin{equation}\label{5isoplin}\per(S)\leq
\frac{R\left(\int_S |C\cc\nn|\,\per+
\nis(\partial S)\right)}{(\DH-1)-R\MS^0}.
\end{equation}

Here, we just remark that there are no closed compact $\HH$-minimal
hypersurfaces immersed in Carnot groups. This fact can be proved
by using the 1st variation formula of the $\HH$-perimeter;
see \cite{27}. The previous formulae have been proved for  
hypersurfaces with boundary, but they hold even if $\partial
S=\emptyset$. More precisely we have:

\begin{Prop}Let $S\subset\GG$ be a compact
 hypersurface of class $\cont^2$  without boundary. Let  $R$ be the radius of the $\varrho$-ball $B_\varrho(0,R)$, centered at the identity $0$ of the group $\GG$ and circumscribed
about $S$.  Then:
\begin{eqnarray}(\DH-1)\,\per(S)&\leq&
R \int_\UU\left(|\MS|+ |C\cc\nn|\right)\per;
\\R&\geq&\frac{(\DH-1)\,\per(S)}{\MS^0\per(S)+\int_S
|C\cc\nn|\,\per}; \\  \per(S)&\leq&
\frac{R \int_S |C\cc\nn|\,\per}{(\DH-1)-R\MS^0}.
\end{eqnarray}

\end{Prop}

\subsection{Application: a weak monotonicity formula}\label{wmf}

In the sequel, we shall set $S_t=S\cap {B_\varrho}(x,t)$. The
``natural'' monotonicity formula which can be deduced from the
 inequality \eqref{1pisoplin} is contained in the next:

\begin{Prop}\label{in}The following inequality
holds \begin{equation}\label{diffin}-\frac{d}{dt}\frac{\per({S}_t)}{t^{\DH-1}}\leq
\frac{1}{t^{\DH-1}}\left(\int_{{S}_t}\left(|\MS|+|C\cc\nn|\right)\per
+ \nis(\partial{S}\cap B_\varrho(x,t))\right)
\end{equation}for $\mathcal{L}^1$-a.e. $t>0$.
\end{Prop}\begin{proof}Since we are assuming that the homogeneous distance $\varrho$ is smooth (at least
piecewise $\cont^1$), by applying the classical Sard's Theorem we
get that ${S}_t$ is a $\cont^2$-smooth manifold with boundary for
$\mathcal{L}^1$-a.e. $t>0$ (or, equivalently, this claim
 follows  by intersecting ${S}$ with the boundary of a
$\varrho$-ball $B_{\varrho}(x,t)$ centered at $x$ and of radius
$t$). So let us apply formula \eqref{jkjlkloko} for the set ${S}_t$.
We have
\[(\DH-1)\,\per({S}_t)\leq t\left(\int_{{S}_t}\big(|\MS|+
|C\cc\nn|\big)\per+ \nis(\partial{S}_t)\right),\]where $t$ is the
radius of a $\varrho$-ball centered at $x$ and intersecting ${S}$.
Since
$$\partial{S}_t=\{\partial{S}\cap B_\varrho(x,t)\}\cup\{\partial
B_\varrho(x,t)\cap {S}\}$$ we get
that
\begin{equation}\label{li}(\DH-1)\,\per({S}_t)\leq
t\left(\underbrace{\int_{{S}_t}\left(|\MS|
+|C\cc\nn|\right)\per}_{:=\mathcal{A}(t)} +
\underbrace{\nis(\partial{S}\cap
B_\varrho(x,t))}_{:=\mathcal{B}(t)} +\nis(\partial
B_\varrho(x,t)\cap {S})\right).\end{equation}Now let us consider
the function $\psi(y):=\|y-x\|_\varrho\,\,\forall\, y\in {S}$. By hypothesis,
$\psi$ is a $\cont^1$-smooth function -at least piecewise-
satisfying $|\grad\cc\psi|\leq 1;$ see Remark \ref{iponhomnor}. So
we may apply the Coarea Formula to this function. Since $|\qq\psi|\leq|\grad\cc\psi|$,  we easily get
that

\begin{eqnarray*}\per({S}_{{t_1}})-\per({S}_{t})&\geq&
\int_{{S}_{{t_1}}\setminus{S}_t}|\qq\psi|\,\per\\&=&\int_{t}^{{t_1}}\nis\{\psi^{-1}[s]\cap
{S}\}\,ds\\&=&\int_{t}^{t_1}\nis(\partial B_\varrho(x,s)\cap
{S})\,ds.\end{eqnarray*}From the last inequality we infer that
$$\frac{d}{dt}\per({S}_t)\geq\nis(\partial B_\varrho(x,t)\cap
{S})$$for $\mathcal{L}^1$-a.e. $t>0$. Hence, from
this inequality and \eqref{li}, we obtain
$$(\DH-1)\,\per({S}_t)\leq t \left(\mathcal{A}(t)+\mathcal{B}(t)+\frac{d}{dt}\per({S}_t)\right) ,$$which is an equivalent form of
\eqref{diffin}.\end{proof}

We have to notice however that, in order to prove an ``intrinsic''
isoperimetric inequality,
 the number $(\DH-1)$ in the previous differential inequality
 {\it is not the correct one}, which   is $(Q-1)$.
 This  fact motivates a further study, made by the author in \cite{25, 26}.

\section{A theorem about non-horizontal graphs in 2-step Carnot groups}\label{heinz}
We begin by describing our result in the simpler setting of the first Heisenberg group $\mathbb H^1$; see also \cite{2332}.
For the notation, see Example \ref{hng}.

\begin{teo}[Heinz's estimate for $T$-graphs]
Let  $S=\left\lbrace p=\exp(x,y,t)\in{\mathbb H}^1:\,t=f(x,
y)\,\,\forall\,(x,y)\in\R^2\right\rbrace$ be a $T$-graph of class
$\mathbf{C}^2$ over the $xy$-plane. If $|\MS|\geq C>0$, then  $$C\,
\mathcal{H}_{Eu}^2(\P_{xy}(\UU))\leq\mathcal{H}_{Eu}^1(\P_{xy}(\partial\UU))$$for
every $\cont^1$-smooth relatively compact open set $\UU\subset S$.
Hence, taking $\UU:=S\cap C_r({\mathcal T})$, where $C_r({\mathcal
T})$ denotes a  vertical cylinder of radius $r$ around the
$T$-axis ${\mathcal T}:=\left\lbrace p=\exp(0,0,t)\in{\mathbb
H}^1,\,t\in\R\right\rbrace$, yields
$$r\leq \frac{2}{C}$$ for every $r>0$.
\end{teo}It follows that any entire $xy$-graph of class $\cont^2$ having constant (or just bounded) horizontal mean curvature $\MS$ must be necessarily  a $\HH$-minimal surface. To see this fact, it is enough to send
$r\longrightarrow + \infty$. The proof of the previous theorem is elementary. More precisely, one uses the following identity:
$$ -\int_{\UU}\MS \varpi\sigma^{2}\cc =\int_{\partial\UU}\nn\LL d\theta,$$where $\theta=T^\ast=dt+\frac{ydx-xdy}{2}$ denotes the dual 1-form to the vertical direction $T$. We also have to remark that $\varpi\sigma^{2}\cc=-d\theta=dx\wedge dy$.
The previous theorem is a generalization to our context of a classical result obtained by Heinz in \cite{112}. This was  generalized by Chern in \cite{933} and then by other authors in a number  of different directions.

Below, we shall restrict ourselves to consider only $2$-step Carnot groups.

\begin{Defi}[Non-horizontal graphs in 2-step Carnot groups] Let $\GG$ be a $2$-step Carnot group and let $Z=\sum_{\alpha\in I\vv}z_\alpha X_\alpha\in\VV$ be a constant vertical vector. In this case, for the sake of simplicity, we reorder the variables in $\gg$ as  $x\equiv(x_{Z^\perp}, x_Z)$, where $x_Z:=\langle x, Z\rangle\in\R$ and $x_{Z^\perp}:=x-x_Z Z\in Z^\perp$. Then, we say that $S\subset\GG$ is a $Z$-graph (over the hyperplane $Z^\perp$) if there exists a function $\psi: Z^\perp\longrightarrow\R$ such that $S=\left\lbrace p=\exp\left(x_{Z^\perp}, \psi(x_{Z^\perp}) \right)\in\GG,\, \, x_{Z^\perp}\in Z^\perp\right\rbrace$.
\end{Defi}

So let us fix  a constant vertical vector $Z\in\VV$ and let
$S=\left\lbrace p=\exp\left(x_{Z^\perp}, \psi(x_{Z^\perp})
\right)\in\GG,\, \, x_{Z^\perp}\in Z^\perp\right\rbrace$ be a
$Z$-graph of class $\mathbf{C}^2$ over the $Z^\perp$-hyperplane.
For the sake of simplicity and without loss of generality, we may
take $Z=X_\alpha$ for a fixed index $\alpha\in I\vv=\left\lbrace
\DH+1,...,n\right\rbrace$.

Now let us define a  differential $(n-2)$-form on $S\subset\GG$ by setting
$$\xi^\alpha:=(\nn\LL X_\alpha\LL\sigma^n\rr)|_{S\setminus C_S} \in\Lambda^2(\TT^\ast S).$$ This differential  $(n-2)$-form $\xi^\alpha$ is well-defined out of $C_S$ and we have to compute its exterior derivative. Below we will briefly sketch a proof, which can also be found in \cite{22}, see Claim 3.22.

\begin{lemma}\label{ibg}We have $d\xi^\alpha|_{S\setminus C_S}=-\MS\varpi_\alpha\per|_{S\setminus C_S}$, at each NC point.
 \end{lemma}

\begin{proof}Let us set $\zeta_j:=(X_\alpha \LL X_j \LL
\Omn)|_S$ for any $\alpha\in I\vv$ and $j\in I\cc$ and compute $d\zeta_j:=d(X_\alpha \LL X_j \LL
\Omn)|_S$. Let $\GG$ be a $k$-step Carnot group. We claim that
\begin{eqnarray}d\zeta_j|_{S\setminus C_S}=
\sum_{k=\alpha+1}^n\,\SC^k_{\alpha j}\,(X_k \LL \Omn)| _{S\setminus C_S} =
\sum_{k=\alpha+1}^n\,\SC^k_{\alpha j}\,\nm_k\, {\sigma^{n-1}\rr}|_{S\setminus C_S}.
\end{eqnarray}
The proof of this claim is just a long, but elementary, calculation. Since we are assuming that $\GG$ has step $2$, using the properties of the Carnot structural constants yields $\SC^k_{\alpha j}=0$ whenever $j, k\in I\cc$ and $\alpha\in I\vv$. Hence $d\zeta_j=0$ for every $j\in I\cc$. By linearity
 $\xi^\alpha=-\sum_{j\in I\cc}\nn^j\zeta_j$, where $\nn^j=\langle\nn,X_j\rangle$ for any $j\in I\cc$. It follows easily that $d\xi^\alpha=-\MS\varpi_\alpha\per$, as wished.
\end{proof}

\begin{teo}[Heinz's estimate for non-horizontal graphs in $2$-step Carnot groups]\label{pdt5}
Let $\GG$ be a $2$-step Carnot group and let $Z\in\VV$ be a
constant vertical vector. Furthermore, let $S$ be a $Z$-graph of
class $\mathbf{C}^2$ over the $Z^\perp$-hyperplane. If $|\MS|\geq
C>0$, then  \begin{equation}\label{mnba} C\,
\Ar(\P_{Z^\perp}(\UU))\leq\mathcal{H}_{Eu}^{n-2}(\P_{Z^\perp}(\partial\UU))
\end{equation}for every $\cont^1$-smooth relatively compact open set $\UU\subset S$. Hence, taking $\UU:=S\cap C_r(\mathcal Z)$, where $C_r({\mathcal Z})$
denotes a Euclidean cylinder of radius $r$ around the $Z$-axis
given  by ${\mathcal Z}:=\left\lbrace
p=\exp(0_{Z^\perp},t)\in\GG,\,t\in\R\right\rbrace$, yields
\begin{equation}\label{mbna}r\leq \frac{n-1}{C}
\end{equation}for every $r>0$.
\end{teo}
\begin{proof} Without loss of generality, we may assume  $-\MS\geq C>0$ and take $Z=X_\alpha$ for some fixed index $\alpha\in I\vv$ . In this case, one has $$\varpi_\alpha\,\per|_S=\nu_\alpha\,\sigma^{n-1}\rr|_S=\left(X_\alpha\LL\sigma^{n}\rr\right) |_S=d\Ar\res X_\alpha^\perp,$$where the last identity follows from our assumption that $S$ is a $X_\alpha$-graph.  By using Lemma \ref{ibg} and Stokes' formula, we obtain the integral identity
$$-\int_\UU\MS\varpi_\alpha\per=\int_{\partial \UU}\nn\LL X_\alpha\LL\sigma^n\rr.$$
Furthermore, we have $$-\int_\UU\MS\varpi_\alpha\,\per=-\int_{\P_{X_\alpha^\perp(\UU)}}\MS d\Ar$$
and  $$\int\left(\nn\LL d\Ar\right)\big|_{\P_{X_\alpha^\perp(\partial \UU)}}=\int\langle\nn,\eta\rangle\,d{\mathcal H}_{Eu}^{n-2}\res\P_{X_\alpha^\perp(\partial \UU)}.$$Putting all together, we get that
$$C\Ar(\P_{X_\alpha^\perp}(\UU))\leq{\mathcal H}_{Eu}^{n-2}\left(\P_{X_\alpha^\perp}(\partial \UU)\right),$$which proves  \eqref{mnba} when $Z=X_\alpha$. Clearly, the thesis follows by linearity. Finally,  \eqref{mbna} follows from \eqref{mnba} and the elementary  calculation  $\frac{{\mathcal H}_{Eu}^{n-2}(\partial B^{n-1}_{Eu})}{\Ar(B^{n-1}_{Eu})}=n-1$, where $B^{n-1}_{Eu}$ denotes a Euclidean unit ball in $Z^\perp\cong\R^{n-1}$.

\end{proof}

It follows that an entire $Z$-graph of class
$\mathbf{C}^2$ over the $Z^\perp$-hyperplane having constant (or bounded) horizontal mean curvature $\MS$ must be necessarily  a $\HH$-minimal hypersurface.

\section{Local Poincar\'{e}-type inequality}\label{wme}
By using an elementary technique, somehow analogous to the one used in Section
\ref{wlinisoineq}, we will state a local Poincar\'e-type
inequality for smooth compactly supported functions  on NC
domains. First we need the following:
\begin{Defi}Let $S \subset\GG$ be a hypersurface of class $\cont^2$ and let $\UU\subseteq S$ be an open domain. We say that $\UU$ is {\rm
uniformly non-characteristic} (abbreviated UNC) if
$$\sup_{x\in\UU}|\varpi(x)|=\sup_{x\in\UU}\frac{|\P\vv\nu(x)|}{|\P\cc\nu(x)|}<+\infty.$$\end{Defi}
We stress that
\begin{eqnarray}\label{biotsav}|C\cc\nn|=\left|\sum_{\alpha\in I\vv}\omega_\alpha C^\alpha\cc\nn\right|\leq
\sum_{\alpha\in I\vv}|\omega_\alpha| \|C^\alpha\cc\|\ngr\leq
\frac{C}{|\PH\nu|},\end{eqnarray}where  $C:=\sum_{\alpha\in I\vv}\|C\cc^\alpha\|\ngr$ only depends on the
structural constants of $\gg$. Let us set
$$R_\UU:=\frac{1}{2\left[\|\MS\|_{L^{\infty}(\UU)} + C\|\varpi\|_{L^{\infty}(\UU)}\right]
}.$$
From
\eqref{biotsav} we have $|C\cc\nn|\leq C\max_{\alpha\in
I\vv}|\varpi_\alpha|$. Moreover
$\int_B|\varpi_\alpha|\,\per=\int_B|\nu_\alpha|\,\sigma^{n-1}\rr\leq\sigma^{n-1}\rr(B)$
for every Borel set $B\subseteq S$.
\begin{teo}\label{0celafo}Let $S \subset\GG$ be a hypersurface of class $\cont^2$. Let $\UU\subset S$ be a  uniformly NC
open domain. Then, for
all $x\in \UU$ and for all $R\leq \min\{{\rm
dist}_\varrho(x,\partial\UU), R_\UU\},$ the following
holds\begin{eqnarray}\label{bibino}
\left(\int_{\UU_R}|\psi|^p\per\right)^{\frac{1}{p}}\leq C_p\,
R\left(\int_{\UU_R}|\qq\psi|^p\per\right)^{\frac{1}{p}} \qquad
p\in[1,+\infty[
\end{eqnarray}for every $\psi\in \cont^1\ss(\UU_R)\cap\cont_0(\UU_R)$. More generally, let $\widetilde{\UU}\subset\UU$
 be a bounded open subset of $\UU$ with smooth boundary and such that ${\rm diam
}_\varrho(\widetilde{\UU})\leq 2 \min\{{\rm
dist}_\varrho(x,\partial\UU), R_\UU\}.$  Then
\begin{eqnarray}\label{bibino2}
\left(\int_{\widetilde{\UU}}|\psi|^p\per\right)^{\frac{1}{p}}\leq
C_p\, {\rm diam
}_\varrho(\widetilde{\UU})\left(\int_{\widetilde{\UU}}|\qq\psi|^p\per\right)^{\frac{1}{p}}
\qquad p\in[1,+\infty[
\end{eqnarray}for every $\psi\in \cont^1\ss(\widetilde{\UU})\cap\cont_0(\widetilde{\UU})$.
\end{teo}

In the above theorem one can take $C_p:=\frac{2p }{2\DH-3}$.

\begin{proof}
 Let
us set
${\psi}_\varepsilon:=\sqrt{\varepsilon^2+\psi^2} \,\,\, (\varepsilon\geq
0)$. By applying Theorem \ref{GD} with $X={\psi}_\varepsilon
x\cc$ we get

\begin{equation*} \int_{\UU_R}\left\{{\psi}_\varepsilon\,\left( (\DH-1)+ \g\MS +  \langle C\cc\nn,
\x\ss\rangle\right)+
\langle\qq{\psi}_\varepsilon,\x\cc\rangle\right\}\,\per =
\int_{\partial\UU_R}{\psi}_\varepsilon\langle
x\cc,\eta\ss\rangle\,\nis,\end{equation*}and so
\begin{eqnarray*}(\DH-1)\int_{\UU_R}{\psi}_\varepsilon\,\per&\leq&
R\left(\int_{\UU_R}\left[{\psi}_\varepsilon\left(|\MS|+
|C\cc\nn|\right) +
|\qq{\psi}_\varepsilon|\right]\per+\int_{\partial\UU_R}{\psi}_\varepsilon\,
\nis\right)\\&\leq&R \left(\|\MS\|_{L^{\infty}(\UU_R)}+ C\|\varpi\|_{L^{\infty}(\UU_R)}
\right)\int_{\UU_R}{\psi}_\varepsilon\,\per\\&+&
R\left(\int_{\UU_R}|\qq{\psi}_\varepsilon|\,\per
+\int_{\partial\UU_R}{\psi}_\varepsilon\,\nis\right).
\end{eqnarray*}By using Fatou's Lemma and the
estimate $R\leq R_\UU$ we get that

\begin{eqnarray*}(\DH-1)\int_{\UU_R}|\psi|\,\per&\leq&(\DH-1)\,
\liminf_{\varepsilon\rightarrow
0^+}\int_{\UU_R}{\psi}_\varepsilon\,\per
\\&\leq&\frac{1}{2}\lim_{\varepsilon\rightarrow
0^+}\int_{\UU_R}{\psi}_\varepsilon\,\per +
R\lim_{\varepsilon\rightarrow
0^+}\left(\int_{\UU_R}|\qq{\psi}_\varepsilon|\,\per
+\int_{\partial\UU_R}{\psi}_\varepsilon\,\nis\right).
\end{eqnarray*}Obviously, $\psi_\varepsilon\longrightarrow|\psi|$ and
$|\qq\psi_\varepsilon|\longrightarrow |\qq\psi|$ as long as
$\varepsilon\rightarrow 0$; moreover $|\psi|=0$ along
${\partial\UU_R}$.  Now since, as it is well-known,
$|\qq|\psi||\leq|\qq\psi|,$ we easily  get the claim by Lebesgue's
Dominate Convergence Theorem. So we have shown that
$$\int_{\UU_R}|\psi|\,\per\leq \frac{2R}{2\DH-3}
\int_{\UU_R}|\qq\psi|\,\per$$ for every  $\psi\in \cont^1\ss(\UU_R)\cap\cont_0(\UU_R)$.
Finally, the general case follows by H\"{o}lder's inequality. More
precisely, let us use the last inequality with $|\psi|$ replaced
by $|\psi|^p$. This implies
\begin{eqnarray*}\int_{\UU_R}|\psi|^p\per&\leq& \frac{2R}{(2\DH-3)}
\int_{\UU_R}p\,|\psi|^{p-1}|\qq\psi|\,\per\\&\leq&\frac{2pR}{(2\DH-3)}\,
\left(\int_{\UU_R}|\psi|^{(p-1)q}\per\right)^{\frac{1}{q}}\left(\int_{\UU_R}|\qq\psi|^{p}\per\right)^{\frac{1}{p}},
\end{eqnarray*}where $\frac{1}{p}+\frac{1}{q}=1$. This achieves the proof of \eqref{bibino}. Finally, \eqref{bibino2} can be proved by repeating
the same arguments as above, just by replacing $R$ with ${\rm
diam}(\widetilde{\UU})$.

\end{proof}

With some extra hypotheses one can show that \eqref{bibino} still
holds up to the characteristic set.

\begin{teo}\label{1celafo}Let $S\subset\GG$ be a hypersurface of class $\cont^2$  with (or without) boundary $\partial S$. We assume that $S$ has bounded horizontal mean curvature $\MS$ and that  $\dim\,C_S<n-2$. Furthermore, let
$\UU_\epsilon\,(\epsilon>0)$ be a family of open subsets of $S$
with   $\cont^1$ boundaries, such
that:\begin{itemize}\item[{\rm(i)}] $C_S\subset\UU_\epsilon$ for
every $\epsilon>0$; \item[{\rm(ii)}]
$\sigma^{n-1}\rr(\UU_\epsilon)\longrightarrow 0$ for
$\epsilon\rightarrow
0^+$;\item[{\rm(iii)}]$\int_{\UU_\epsilon}|\PH\nu|\,\sigma^{n-2}\rr\longrightarrow
0$ for $\epsilon\rightarrow 0^+$.\end{itemize} Then, for
every $x\in S$  and
 every (small enough) $\epsilon>0$ there exists $R_0:=R_0(x, \epsilon)\leq {\rm dist}_\varrho(x,\partial
S)$ such that\begin{eqnarray}\label{bibino1}
\left(\int_{S_R}|\psi|^p\per\right)^{\frac{1}{p}}\leq C_p\,
R\left(\int_{S_R}|\qq\psi|^p\per\right)^{\frac{1}{p}} \qquad
p\in[1,+\infty[
\end{eqnarray}
holds for every $\psi\in\cont\ss^1(S_R)\cap \cont_0(S_R)$ and every
$R\leq R_0$, where
$$R_0:=\min\left\{{\rm dist_\varrho}(x,\partial S),\,\,
\frac{1}{2\left[C\left(1+
\|\varpi\|_{L^{\infty}(S_R\setminus\UU_{\epsilon})} \right) +
\|\MS\|_{L^{\infty}(S_R)} \right]} \right\}.$$
\end{teo}

\begin{proof} Set
$\psi_\varepsilon:=\sqrt{\varepsilon^2+\psi^2}\,\,(0\leq\varepsilon<1).$
We shall  prove the theorem for $p=1$. The  general case will
follow by using H\"{o}lder's inequality. Let
$\UU_\epsilon\,(\epsilon>0)$ be as above. Fix $\epsilon_0>0$.
For every $\epsilon\leq \epsilon_0$ one has
$$\int_{\UU_{\epsilon}}\psi_\varepsilon|C\cc\nn|\,\per\leq 2  C\,\|\psi\|_{L^{\infty}(\UU_{\epsilon_0})}\sigma^{n-1}\rr(\UU_\epsilon),$$where we have put $C:=\sum_{\alpha\in I\vv}\|C\cc^\alpha\|\ngr$.
Furthermore (ii) implies that for every $\delta>0$ there exists
$\epsilon_\delta>0$ such that
$\sigma^{n-1}\rr(\UU_\epsilon)<\delta$ whenever
$\epsilon<\epsilon_\delta$. Taking $\widetilde{\delta}\leq
\frac{\int_{S_R}\psi_\varepsilon\,\per}{2\|\psi\|_{L^{\infty}
(\UU_{\epsilon_{0}}) }}$, one gets
$$\int_{\UU_{\epsilon}}\psi_\varepsilon|C\cc\nn|\,\per\leq
 C\int_{S_R}\psi_\varepsilon\,\per$$ for every $\epsilon\leq \min\{
\epsilon_{\widetilde{\delta}}, \epsilon_0\}$. Moreover, for any
  $\epsilon\in]0,\min\{ \epsilon_{\widetilde{\delta}},
\epsilon_0\}[$, one has
 $$\int_{S_R\setminus\UU_{\epsilon}}\psi_\varepsilon|C\cc\nn|\,\per\leq C\|\varpi\|_{L^{\infty}
(S_R\setminus\UU_{\epsilon})}\int_{S_R}\psi_\varepsilon\,\per.$$
It follows that
$$\int_{S_R}\psi_\varepsilon|C\cc\nn|\,\per\leq C\big(1+ \|\varpi\|_{L^{\infty}
(S_R\setminus\UU_{\epsilon})}
\big)\int_{S_R}\psi_\varepsilon\,\per.$$ Since, by hypothesis, the
horizontal mean curvature is bounded, we clearly have
$$\int_{S_R}\psi_\varepsilon|\MS|\,\per\leq \|\MS\|_{L^{\infty}
(S_R)}\int_{S_R}\psi_\varepsilon\,\per.$$ Applying Theorem \ref{GD} with $X=\psi_\varepsilon x\cc$ (and arguing as in the proof of Theorem \ref{0celafo}) yields \begin{eqnarray*}(\DH-1)\int_{S_R}{\psi}_\varepsilon\,\per&\leq&
R\left(\int_{S_R}\big\{{\psi}_\varepsilon\big(|\MS|+ |C\cc\nn|\big)
+ |\qq{\psi}_\varepsilon|\big\}\per+\int_{\partial
S_R}{\psi}_\varepsilon \nis\right)\\&\leq&R\big[C\big(1+
\|\varpi\|_{L^{\infty}(S_R\setminus\UU_{\epsilon})} \big) +
\|\MS\|_{L^{\infty}(S_R)}\big]\int_{S_R}\psi_\varepsilon\,\per \\
&+& R\left(\int_{S_R}|\qq{\psi}_\varepsilon| \per+\int_{\partial
S_R}\psi_{\varepsilon}\nis\right).\end{eqnarray*}So if $R\leq R_0$, one gets
$$\int_{S_R}{\psi}_\varepsilon\,\per\leq \frac{2R}{2\DH-3}\left(\int_{S_R}|\qq\psi_\varepsilon|\,\per
+ \int_{\partial S_R}\psi_{\varepsilon}\nis\right).$$ We have
$\psi_\varepsilon\longrightarrow|\psi|$ and
$|\qq\psi_\varepsilon|\longrightarrow |\qq\psi|$ as long as
$\varepsilon\rightarrow 0$ and  $|\psi|=0$ along ${\partial S_R}$.
Since $|\qq|\psi||\leq|\qq\psi|,$ the thesis follows from  Fatou's lemma and Lebesgue's Dominated Convergence Theorem.
\end{proof}

\subsection{A Caccioppoli-type inequality}\label{Cacciopp}
Our final result is a generalization of the classical \it Caccioppoli inequality \rm (see, for instance, \cite{Ambns}) for the operator $\lll$ on smooth hypersurfaces.

Let $S\subset\GG$ be a hypersurface of class $\cont^2$ and  
set $S_R:=S\cap B_\varrho(x, R)$ for any $x\in \GG$. We are going
to consider the functions satisfying, in the distributional sense,
the following problem:
\begin{equation}\label{ci}-\lll\phi=\psi
\qquad\mbox{on}\,\,S_R,\end{equation}whenever $\psi\in L^2(S_R, \per)$.

So let us take a function $\zeta\in \cont^1\ss(S_R)\cap\cont_0(S_R)$
such that $0\leq \zeta\leq 1$, $\zeta=1$ on $S_{R/2}=S\cap
B_\varrho(0,R/2)$ and $|\grad\ss\zeta|\leq C_0/R$. Inserting
into the above equation the function
$\varphi=\zeta^2(\phi-\phi_0)$, where $\phi_0\in\R$ is a fixed constant, and then integrating
over $S_R$, yields

\begin{eqnarray*}\underbrace{\int_{S_R}\zeta^2|\qq\phi|^2\,\per}_{:=I_1}+ \underbrace{2\int_{S_R}
\zeta
(\phi-\phi_0)\langle\qq\zeta,\qq\phi\rangle\,\per}_{:=I_2}=
\underbrace{\int_{S_R}\psi\zeta^2(\phi-\phi_0)\,\per}_{:=I_3}.\end{eqnarray*}
We
have $$I_2\leq \frac{1}{2}\int_{S_R}|\zeta|^2|\qq\phi|^2\,\per +
\underbrace{2 \int_{S_R}|\phi-\phi_0|^2|\qq\phi|^2\,\per}_{:=I_4}.
$$Moreover $I_4\leq {2
C_0^2}/{R^2}\|\phi-\phi_0\|_{L^2(S_R)}$.  Now let us estimate the third integral
$I_3$. We have
\begin{eqnarray*}\int_{S_R}\psi\zeta^2(\phi-\phi_0)\,\per&=&
\int_{S_R}2\left((2R\psi)\frac{\zeta^2(\phi-\phi_0)}{4R}\right)\per
\\&\leq& 4R^2\int_{S_R}\psi^2\,\per +
\frac{1}{16R^2}\int_{S_R}\zeta^4|\phi-\phi_0|^2\,\per\\&\leq&
4R^2\int_{S_R}2\psi^2\,\per +
\frac{1}{R^2}\int_{S_R}|\phi-\phi_0|^2\,\per.\end{eqnarray*}
Since $\zeta=1$ on $S_{R/2}$, using the previous estimates yields\[\int_{S_{R/2}}|\qq\phi|^2\,\per \leq
\frac{2C_0^2+1}{R^2}\int_{S_{R}}|\phi-\phi_0|^2\,\per +
4R^2\int_{S_{R}}\psi^2\,\per.\]

We summarize these calculations, as follows:
\begin{teo}Let $S\subset\GG$ be a hypersurface of class $\cont^2$; let $\phi_0\in\R$ and let $\phi$ be a distributional solution to the equation $-\lll\phi=\psi
\,\,\,\mbox{on}\, \, S_R,$ where $\psi\in L^2(S_R, \per)$. Then, there exists a positive constant $C>0$ such that the following ``Caccioppoli-type'' inequality holds:
 \[\int_{S_{R/2}}|\qq\phi|^2\,\per \leq C\left(
\frac{1}{R^2}\int_{S_{R}}|\phi-\phi_0|^2\,\per +
R^2\int_{S_{R}}\psi^2\,\per\right)\]for every (small enough)
$R>0$, where $S_R:=S\cap
B_\varrho(x,R)$, for any $x \in S$. \end{teo}

\bibliographystyle{alpha}

{\scriptsize \noindent Francescopaolo Montefalcone:
\\Dipartimento di Matematica \\Universit\`a degli Studi di Padova,\\
  Via Trieste, 63, 35121 Padova (Italy)\,
 \\ {\it E-mail address}:  {\textsf montefal@math.unipd.it}}

\end{document}